\documentclass[12pt,authoryear,letterpaper,reqno,oneside]{amsart}
\pdfoutput=1
\usepackage[left=1.52in,right=1.52in,bottom=1.52in,top=1.52in]{geometry}
\usepackage{amsmath,amsfonts,amssymb}%
\usepackage[usenames]{color}
\usepackage{setspace}
\usepackage{comment}
\usepackage[colorlinks,citecolor=darkblue,urlcolor=darkblue,linkcolor=darkgreen]{hyperref}
\usepackage{stmaryrd, bbm}
\usepackage{relsize}

\usepackage{microtype}

\newtheorem{theorem}{Theorem}[section]%
\newtheorem{proposition}[theorem]{Proposition}%
\newtheorem{lemma}[theorem]{Lemma}%
\newtheorem{corollary}[theorem]{Corollary}%
\newtheorem{definition}[theorem]{Definition}%

\DeclareMathOperator{\Str}{Str}
\DeclareMathOperator{\Th}{Th}

\DeclareMathOperator{\Var}{Var}

\renewcommand{\hat}{\widehat}

\newcommand{\Ind}{{\mathbbm{1}}}
\newcommand{\fkn}{{\mathfrak{n}}}

\def\st{\,:\,}

\def\M{{\EM{\mathcal{M}}}}

\def\cM{{\EM{\mathcal{M}}}}
\def\cD{{\EM{\mathcal{D}}}}

\def\N{{\EM{\mathcal{N}}}}
\def\cN{{\EM{\mathcal{N}}}}

\def\Pr{{\EM{\mathbb{P}}}}
\def\cA{{\EM{\mathcal{A}}}}
\def\cB{{\EM{\mathcal{B}}}}

\def\F{{\EM{\mc{F}}}}

\def\x{{\EM{\ol{x}}}}
\def\xx{{\EM{\ol{x}}}}

\def\a{{\EM{\ol{a}}}}

\def\m{{\EM{\textbf{m}}}}

\def\Lang{{\EM{L}}}

\def\Lomega#1{{\EM{\mc{L}_{#1, \w}}}}

\def\Lww{\Lomega{\w}}
\def\Lwow{\Lomega{\w_1}}

\newcommand{\NonRepeating}{\mathrm{NR}}

\newcommand{\UniformMeasure}[1]{\mu_{\mathrm{un},#1}}
\newcommand{\UniformMeasureLang}{\UniformMeasure{\Lang}}

\newcommand{\Sym}[1]{\mathrm{Sym}(#1)}
\newcommand{\Entropy}[1]{\mathrm{Ent}(#1)}

\newcommand{\EntropyFin}[1]{h(#1)}
\newcommand{\EntropyFinbig}[1]{h\bigl(#1\bigr)}
\newcommand{\EntropyFinBig}[1]{h\Bigl(#1\Bigr)}

\newcommand{\RandomEntropyFin}[1]{H(#1)}
\newcommand{\RandomEntropyFinbig}[1]{H\bigl(#1\bigr)}
\newcommand{\RandomEntropyFinBig}[1]{H\Bigl(#1\Bigr)}

\newcommand{\Wsqsq}{W_{\text{\larger[-2]$\blacksquare \square$}}}
\newcommand{\Wbsq}{W_{\text{\larger[-2]$\blacksquare$}}}
\newcommand{\Wwsq}{W_{\text{\larger[-2]$\square$}}}

\newcommand{\Thnr}{\mathrm{Th}_\mathrm{nr}}

\newcommand{\Thfun}{\mathrm{Th}_\mathrm{sel}}

\newcommand{\Equiv}{E}

\newcommand{\defn}[1]{{\bf{#1}}}
\newcommand{\defas}{{\EM{\ :=\ }}}

\newcommand{\llrr}[1]{{\llbracket #1 \rrbracket}}

\newcommand{\Extent}[1]{{\llrr{#1}}}

\newcommand{\Exp}{{\EM{\mathbb{E}}}}

\newcommand{\ProbMeas}{{\EM{\mathcal{P}}}}

\newcommand{\ER}{Erd\H{o}s--R\'enyi}

\newcommand{\Rado}{\mathcal{R}}

\def\w{\EM{\omega}}

\def\Reals{{\EM{{\mbb{R}}}}}

\def\Nats{{\EM{{\mbb{N}}}}}

\def\^{\EM{{}^{\And}}}

\def\And{\EM{\wedge}}

\def\<{\EM{\langle}}
\def\>{\EM{\rangle}}

\def\EM#1{\ensuremath{#1}}
\def\mbb#1{\EM{\mathbb{#1}}}

\def\mc#1{\EM{\mathcal{#1}}}

\def\ol#1{\EM{\overline{#1}}}

\newcommand{\qf}{\EM{\mathfrak{qf}^{\mathrm{nr}}}}

\newcommand{\arity}{\EM{\mathrm{arity}}}
\newcommand{\ArityLang}{\mathfrak{a}_{\Lang}}
\newcommand{\Alice}{A}

\newcommand{\Powerset}{\ensuremath{\mathfrak{P}}}

\def\given{\ |\ }
\def\givenbig{\ \big|\ }

\def\givenBig{\ \Big|\ }

\definecolor{MyGreen}{rgb}{.75,0,.75}
\definecolor{RealGreen}{rgb}{0,1,0}
\definecolor{DarkGreen}{rgb}{0.1,0.4,0.1}
\definecolor{ActualGreen}{rgb}{0.0,0.5,0.0}
\definecolor{MyBlue}{rgb}{0,0,1}
\definecolor{MyRed}{rgb}{1,0,0}

\definecolor{darkred}{rgb}{0.5,0,0}
\definecolor{darkgreen}{rgb}{0, 0.3,0}
\definecolor{darkblue}{rgb}{0,0,0.6}

\definecolor{LightGray}{rgb}{.6,.6,.6}

\begin{document}


\title{The entropy function of an invariant measure}

\author[Ackerman]{Nathanael Ackerman}
\address{
Department of Mathematics\\
Harvard University\\
Cambridge, MA 02138
}
\email{nate@math.harvard.edu}

\author[Freer]{Cameron Freer}
\address{
Borelian Corporation\\
Cambridge, MA 02139
}
\email{freer@borelian.com}

\author[Patel]{Rehana Patel}
\address{
Department of Philosophy\\
Harvard University\\
Cambridge, MA 02138
}
\email{rpatel@fas.harvard.edu}

\begin{abstract}
	Given a countable relational language $\Lang$, 
we consider probability measures on the space of $L$-structures with underlying set $\Nats$ 
	that are invariant under the logic action. We study the growth rate of the entropy function of such a measure, defined to be 
	the function sending $n \in \Nats$ to
	the entropy of the measure induced by restrictions to $L$-structures on $\{0, \ldots, n-1\}$.
	When $\Lang$ has finitely many relation symbols, all of arity $k\ge1$, and the measure has  a property called non-redundance, we show that the entropy function is of the form $Cn^k+o(n^k)$, generalizing a result of Aldous and Janson.  When $k\ge2$, we show that there are invariant measures whose entropy functions grow arbitrarily fast in $o(n^k)$, extending a result of Hatami--Norine.
	For possibly infinite languages $L$, we	give an explicit upper bound on the entropy functions of non-redundant invariant measures in terms of the number of relation symbols in $\Lang$ of each arity; this implies that finite-valued entropy functions can grow arbitrarily fast.  
\end{abstract}

\maketitle


\setcounter{page}{1}
\thispagestyle{empty}

\begin{small}
\begin{tiny}
\renewcommand\contentsname{\!\!\!\!}
\setcounter{tocdepth}{3}
\tableofcontents
\end{tiny}
\end{small}

\section{Introduction}
\label{intro-sec}

The information-theoretic notion of entropy 
captures the idea of the typical ``uncertainty'' in 
a sample of a given probability measure. 
An important class of probability measures 
in model theory, combinatorics, and probability theory
are the $\Sym{\Nats}$-invariant measures on the space of structures, in a countable language, with underlying set $\Nats$. We call such measures \emph{invariant measures},
and consider the entropy of such measures.

When the support of a probability measure is uncountable, as is usually the case for invariant measures,
the entropy of the measure is infinite.
But any such measure $\mu$ can be approximated by its projections to the spaces of structures on initial segments of $\Nats$. 
We define the
\emph{entropy function} of $\mu$ to be the
function from $\Nats$ to $\Reals \cup \{\infty\}$ taking $n$ 
to the entropy of the measure 
induced by $\mu$
on structures with underlying set $\{0, \dots, n-1\}$.
In this paper, we study the growth of entropy functions.

The growth of the entropy function of an invariant measure
has been
studied by Aldous \cite[Chapter~15]{MR883646} in the setting of exchangeable arrays and
by Janson \cite[\S10 and \S D.2]{MR3043217}
and Hatami--Norine \cite[Theorem~1.1]{MR3073488}
in the case where the measure is concentrated on the space of graphs.
The leading coefficient of the entropy function has been
used to study large deviations for \ER\ random graphs and exponential random graphs in 
\cite{MR2825532} and
\cite{MR3127871}, and phase transitions in exponential random graphs in \cite{MR3083277}.
For additional results on the entropy functions of invariant measures concentrated on the space of graphs, see \cite{MR3742179}.

Further, ergodic invariant measures are the focus of a project that 
treats them model-theoretically as a notion of ``symmetric probabilistic structures''; see
\cite{MR3515800},
\cite{MR3425980}, 
\cite{MR3564374}, 
\cite{AFPcompleteclassification}, and
\cite{2017arXiv171009336A}. 
The entropy function of such a measure is one gauge of its complexity.

In this paper we primarily consider invariant measures, for a countable relational language, that
are \emph{non-redundant}, namely those that concentrate on structures in which a relation can hold of a tuple only when the tuple consists of distinct elements.
In \S\ref{qf-interdef-subsec} we develop machinery for quantifier-free interdefinitions that allows us to show, in
       Proposition~\ref{nonredundant-entropy-function}, that every entropy function of an invariant measure for a countable language is the entropy function of some non-redundant invariant measure for a countable relational language.

In Section~\ref{section-finite}, we study the case of invariant measures for countable relational languages
$\Lang$ with all relations having the same arity $k\ge 1$. 
The
Aldous--Hoover--Kallenberg theorem 
provides a 
representation, which we call an \emph{extended $\Lang$-hypergraphon}, of a non-redundant invariant measure for $\Lang$.
Our main technical result in this section,
Theorem~\ref{Entropy-from-hypergraphon},
shows that under a certain condition (which is satisfied, e.g., when $L$ is finite), the entropy function of the invariant measure corresponding to an extended $\Lang$-hypergraphon $W$
grows like $C n^k + o(n^k)$, where the constant $C$ can be calculated from $W$.
When the invariant measure is concentrated on the space of graphs, 
Theorem~\ref{Entropy-from-hypergraphon}
reduces to a result first observed by Aldous in \cite[Chapter~15]{MR883646}, and also by Janson in \cite[Theorem~D.5]{MR3043217}.
The proof of
Theorem~\ref{Entropy-from-hypergraphon}
follows closely
that of Janson.

In Section~\ref{randomfree-sec} we 
consider invariant measures that arise
via sampling from a Borel $k$-uniform hypergraph for $k\ge2$.
The entropy functions of such measures grow like $o(n^k)$,
	as we show in Lemma~\ref{random-free-has-leading-entropy-zero}.
We prove, in Theorem~\ref{Random-free has arbitrary growth rates},
that 
for every function $\gamma(n) \in o(n^k)$ 
there is an invariant measure sampled from
a Borel 
$k$-uniform hypergraph whose entropy function grows faster than $\gamma(n)$. Our 
theorem
generalizes a result of Hatami--Norine \cite[Theorem~1.1]{MR3073488},
who prove it in the case where $k  = 2$ (hence where the Borel hypergraph is simply a Borel graph). 
Hatami--Norine's proof analyzes the random graph obtained by subsampling from an appropriate ``blow up'' of 
the Rado graph, i.e., the unique countable homogeneous-universal
graph.
Our proof takes an analogous path, 
using the unique countable homogeneous-universal 
$k$-uniform hypergraph.

Finally, in Section~\ref{nonredundant-sec} we 
consider invariant measures for countable languages that may be of unbounded arity.
We provide
an upper bound on the entropy functions of 
non-redundant invariant measures for a given countable relational language,
in terms of the number of relations of each arity. 
We do so by calculating the entropy function of a particular non-redundant invariant measure whose entropy function is maximal among 
those
for that language.
This calculation also demonstrates that
whereas the growth of the entropy function of an invariant measure for a finite language is polynomial, there are $\Reals$-valued entropy functions
that grow arbitrarily fast.

\subsection{Preliminaries}
\emph{In this subsection, 
let $\Lang$ be a countable language and
let $\fkn\in\Nats \cup \{\Nats\}$.}

For $n\in\Nats$, write $[n]\defas \{0, \ldots, n-1\}$, and $[\Nats] \defas \Nats$. 
Let $\Sym{\fkn}$ denote the collection of permutations of $[\fkn]$.

For $k \in \Nats$, 
define $\Powerset_{<k}(\fkn) \defas \{Y \subseteq [\fkn] \st |Y| < k\}$ and $\Powerset_{k}(\fkn) \defas$ $\{Y \subseteq [\fkn] \st |Y| = k\}$, and let
$\Powerset_{\leq k}(\fkn) \defas
\Powerset_{< k}(\fkn) \cup
\Powerset_{k}(\fkn)$.
We order each of these sets using \emph{shortlex} order, i.e., ordered by size, with sets of the same size ordered lexicographically.

We write $\Lww(\Lang)$ to denote the collection of first-order $\Lang$-formulas.
An \emph{$\Lang$-theory} is a collection of first-order $\Lang$-sentences.
A theory is $\Pi_1$ when every sentence is quantifier-free or of the form $(\forall \x)\varphi(\x)$ where $\x$ is a tuple of variables and $\varphi$ is quantifier-free.
The \emph{maximum arity} of $L$ is the maximum arity, when it exists, of a relation symbol or function symbol in $L$, where we consider constant symbols to be function symbols of arity $0$.

Fix a probability space $(\Omega, \F, \Pr)$.
Suppose $(D, \cD)$ is a measurable space. A \emph{$D$-valued random variable} $Z$, also called a \emph{random element in} $D$, is an $(\F,\cD)$-measurable function $Z \colon \Omega \to D$. 
The \emph{distribution} of $Z$ is the 
probability measure $\Pr\circ Z^{-1}$.
Given an event $B \in \cD$, we say that $B$ holds \emph{almost surely} when $\Pr(B) = 1$, and abbreviate this \emph{a.s.}  Typically $B$ will be specified indirectly via some property of random variables; for example, we say that  two random variables are equal a.s.\ when the subset of $\Omega$ on which they are equal has full measure.
For a topological space $S$, let $\ProbMeas(S)$ denote the space of Borel probability measures on $S$, with $\sigma$-algebra given by the weak topology.
We use the symbol $\bigwedge$ for conjunctions of probabilistic events, as well as for conjunctions of logical formulas.
	See \cite{MR1876169} for further background and notation from probability theory.

	We write $\lambda$ to denote the uniform (Lebesgue) measure on $[0, 1]$ and on finite powers of $[0,1]$.
Suppose a variable $x$ takes values in some finite power of $[0,1]$.
We say that an expression involving $x$ holds \emph{almost everywhere}, abbreviated \emph{a.e.}, when it holds of $x$ on all but a $\lambda$-null subset.

Let $\Str_\Lang(\fkn)$
denote the collection of $\Lang$-structures that have underlying set $[\fkn]$. 
We will consider $\Str_\Lang(\fkn)$ as a measure space with the $\sigma$-algebra generated by the topology given by basic clopen
sets of the form 
\[
\{\cM \in \Str_\Lang(\fkn) \st \cM \models \varphi(x_0, \dots, x_{\ell-1})\}
\]
when $\varphi\in \Lww(L)$ is a quantifier-free formula with $\ell$ free variables and $x_0, \dots, x_{\ell-1} \in [\fkn]$.

\begin{definition}
	Suppose $\varphi \in \Lww(\Lang)$ has $\ell$ free variables and $r_0, \dots, r_{\ell-1} \in [\fkn]$.
	Define the \defn{extent} (on $[\fkn]$) of $\varphi(r_0, \dots, r_{\ell-1})$ to be 
\[
	\Extent{\varphi(r_0, \dots, r_{\ell-1})}_{[\fkn]} \defas \{ \M \in \Str_\Lang(\fkn) \st \M \models \varphi(r_0, \dots, r_{\ell-1})\}.
\]
	When $\fkn=\Nats$ we will sometimes omit the subscript $[\fkn]$.
	For an $\Lang$-theory $T$, define $\Extent{T}_{[\fkn]} \defas \bigcap_{\rho\in T}\Extent{\rho}_{[\fkn]}$.
\end{definition}

There is a natural action of $\Sym{\fkn}$ on $\Str_\Lang(\fkn)$ called \defn{the logic action},
defined as follows:
	 for $\sigma\in\Sym{\fkn}$ and $\M\in\Str_\Lang(\fkn)$, let
	  $\sigma \cdot \M$ be the structure $\N\in\Str_\Lang(\fkn)$ for which
	    \[
			 R^\N(r_0, \ldots, r_{k-1})
			  \quad \text{if and only if} \quad
			   R^\M\bigl(\sigma^{-1}(r_0), \ldots,  \sigma^{-1}(r_{k-1}) \bigr)
			    \]
				 for all relation symbols $R\in L$ and $r_0, \ldots, r_{k-1} \in\fkn$, where $k$ is the arity of $R$, and similarly with constant and function symbols.
				  Note that the orbit 
				  under the logic action 
				  of any structure 
				  in $\Str_\Lang(\fkn)$
				  is its isomorphism class.
				  By Scott's isomorphism theorem,
				  every such orbit is Borel. 
				   For more details on the logic action, see \cite[\S16.C]{MR1321597}.

We say a probability measure $\mu$ on $\Str_{\Lang}(\fkn)$ is \defn{invariant} if it is invariant under the logic action of $\Sym{\fkn}$, i.e., if 
$\mu(B) = \mu(\sigma \cdot B)$
for every Borel $B \subseteq \Str_\Lang(\fkn)$ and every $\sigma \in\Sym{\fkn}$.
We call such a probability measure an \emph{invariant measure for $\Lang$}. 
An invariant measure $\mu$ is \defn{ergodic} if
	$\mu(X) = 0$ or $\mu(X) = 1$
	whenever $\mu(X\triangle\sigma(X)) = 0$ for all $\sigma \in\Sym{\Nats}$.
	Every ergodic invariant measure 
	on $\Str_\Lang(\fkn)$
	is an \emph{extreme point} in the simplex of invariant measures 
	on $\Str_\Lang(\fkn)$, and any invariant measure 
	on $\Str_\Lang(\fkn)$
	can be decomposed as a mixture of ergodic ones (see
	\cite[Lemma~A1.2 and Theorem~A1.3]{MR2161313}).

For $n\in\Nats$, any probability measure $\mu$ on $\Str_\Lang(\Nats)$ induces a probability measure $\mu_n$ on $\Str_\Lang(n)$ such that for any Borel set $B \subseteq  \Str_\Lang(n)$, we have 
\[
\mu_n(B) \defas \mu(\{\cM \in \Str_\Lang(\Nats) \st \cM|_{[n]} \in B\}),
\]
where $\cM|_{[n]}\in\Str_\Lang(n)$ denotes the restriction of $\cM$ to $[n]$.
Further, by the Kolmogorov consistency theorem, $\mu$ is uniquely 
determined by the collection $\<\mu_n\>_{n \in \Nats}$.

Towards defining the entropy function of an invariant measure, we give the standard definition of the entropy of a probability measure. For this definition, we use the convention that $-0 \log_2(0) = 0$. 

\begin{definition}
\label{entropy-definition-finite-set}
	Let $\nu$ be a probability measure on a standard Borel space $S$, and let 
	$A\defas	\{s\in S \st \nu(\{s\}) > 0\}$ be its (countable) set of atoms.
	If $\nu$ is purely atomic, i.e., $\nu(A) = 1$,
	then the \defn{entropy} of $\nu$ is given by
\[
\EntropyFin{\nu}\defas -\sum_{x \in A} 
\nu(\{x\})  \log_2(\nu(\{x\})). 
\]
Otherwise, let
\[
\EntropyFin{\nu}\defas \infty.
\]
For any random variable $X$ with distribution $\nu$, define $\EntropyFin{X} \defas \EntropyFin{\nu}$.
\end{definition}

\begin{definition}
	The \defn{joint entropy} of a pair of random variables $X$ and $Y$, written
	$h(X,Y)$, is defined to be the entropy of the joint distribution of $(X,Y)$. 
	Similarly, 
	$h(\<X_i\>_{i\in I})$ 
	is defined to be the entropy of the joint distribution of the sequence $\<X_i\>_{i\in I}$.
\end{definition}

\begin{definition}
	Given random variables $X$ and $Y$, the \defn{conditional entropy} of $X$ given $Y$, written
	$h(X\in \cdot\ |\,Y)$, is defined to be the function
	\[
	y \mapsto 
		h(\Pr(X\in\cdot\ |\,Y = y)).\]
\end{definition}

We have defined the entropy of a random variable to be the entropy of its distribution. When the random variable takes values in a space of measures, instead of considering the entropy of the random variable directly, we sometimes need the random variable, defined below, that is obtained by taking the entropies of these measures themselves.

\begin{definition}
		Suppose $\chi$ is a measure-valued random variable.
	The \defn{random entropy} of $\chi$, written
	$H(\chi)$, is a random variable defined by
	\[
		H(\chi)(\varpi) = h(\chi(\varpi)).
		\]
		for $\varpi\in\Omega$.
\end{definition}

This notion allows us to define random conditional entropy.

\begin{definition}
	Given random variables $X$ and $Y$, the \defn{random conditional entropy} of $X$ given $Y$, written
	$H(X\,|\,Y)$, is defined by
	\[ H(X\,|\,Y) = H(\Exp(X\,|\,Y)),
		\]
		i.e., the random entropy of the random measure $\Exp(X\,|\,Y)$, the conditional expectation of $X$ given $Y$.
\end{definition}

The following three lemmas are standard facts about entropy, which we will need later.

\begin{lemma}[\textnormal{see {\cite[Theorem~2.2.1]{MR2239987}}}]
\label{chainrule}
For random variables $X$ and $Y$, we have
\[
	\EntropyFin{X, Y} = \EntropyFin{X} + \Exp(\RandomEntropyFin{Y \given X}).
\]
\end{lemma}

\begin{lemma}[\textnormal{see {\cite[Theorem~2.6.5]{MR2239987}}}]
\label{cond-reduces-entropy}
For random variables $X$ and $Y$, we have
		\[ \EntropyFin{X} \ge \Exp(\RandomEntropyFin{X \given Y}),\]
	with equality if and only if $X$ and $Y$ are independent.
\end{lemma}

\begin{lemma}[\textnormal{see {\cite[Theorem~2.6.6]{MR2239987}}}]
\label{joint-entropy-less-than-sum-entropy}
Let $\<X_i\>_{i \in I}$ be a sequence of random variables. Then 
\[
\EntropyFin{\<X_i\>_{i \in I}} \leq \sum_{i \in I} \EntropyFin{X_i},
\]
with equality if and only if
the $X_i$ are independent.
\end{lemma}

We now define the entropy function of an invariant measure on $\Str_\Lang(\Nats)$.

\begin{definition}
Let $\mu$ be an invariant measure on $\Str_\Lang(\Nats)$. The \defn{entropy function} of $\mu$ is defined to be the function $\Entropy{\mu}$ from $\Nats$ to $\Reals \cup \{\infty\}$ given by 
\[
\Entropy{\mu}(n) \defas \EntropyFin{\mu_n}.
\]
\end{definition}

The following basic property of the entropy function is immediate.

\begin{lemma}
Let $\mu$ be an invariant measure on $\Str_\Lang(\Nats)$. Then $\Entropy{\mu}$ is a non-decreasing function.
\end{lemma}

In this paper we will mainly be interested in invariant measures whose entropy functions take values in $\Reals$. 
The following lemma is immediate because
in a finite language,
there are only finitely many structures of each finite size.

\begin{lemma}
	\label{realvalued-lemma}
Suppose $\Lang$ is finite, and let $\mu$ be an invariant measure on $\Str_\Lang(\Nats)$. Then 
$\Entropy{\mu}$ 
	is $\Reals$-valued. 
\end{lemma}

In Section~\ref{nonredundant-sec} we will strengthen this lemma by showing that there is a polynomial upper bound of $O(n^k)$ 
on $\Entropy{\mu}(n)$, where $k$ is the maximum arity of $\Lang$.

The notion of \emph{non-redundance} 
for an invariant measure, which we introduce next, will be key throughout this paper.

\begin{definition}
	\label{Thnr-def}
Suppose $\Lang$ is relational.
For each $R \in \Lang$ 
	define $\vartheta_R$ to be the formula
\[
(\forall x_0, \dots, x_{k-1}) \Big(R(x_0, \dots, x_{k-1}) \rightarrow \bigwedge_{i < j < k} (x_i \neq x_j)\Big),
\]
where $k$ is the arity of $R$.
Define the $\Pi_1$ $\Lang$-theory
	\[
		\Thnr(\Lang) \defas \{\vartheta_R \st R \in\Lang\}.\]

An $\Lang$-structure $\M$ is \defn{non-redundant} when $\M \models \Thnr(\Lang)$. An invariant measure $\mu$ on $\Str_\Lang(\Nats)$ is \defn{non-redundant} when 
\[
\mu(\Extent{\Thnr(\Lang)}) = 1.
\]
\end{definition}

For example, any $k$-uniform hypergraph is non-redundant, as is any directed graph without self-loops.

The following straightforward lemma provides conditions under which the entropy function takes values in $\Reals$.

\begin{lemma}
	\label{non-redundant-makes-it-real-valued}
Suppose $\Lang$ 
has finitely many relations of any given arity, and let $\mu$ be a non-redundant invariant measure on $\Str_{\Lang}(\Nats)$. Then for every $n \in \Nats$ there is a finite set $A_n \subseteq \Str_{\Lang}(n)$ such that $\mu_n$ concentrates on $A_n$. 
In particular, 
	$\Entropy{\mu}(n)\in\Reals$,
	and so
	$\Entropy{\mu}$ is $\Reals$-valued.
\end{lemma}

In Section~\ref{nonredundant-sec} we will consider the case of non-redundant invariant measures for a relational language, and will strengthen this lemma by providing an explicit upper bound on $\Entropy{\mu}(n)$.


\subsection{Quantifier-free interdefinitions}
\label{qf-interdef-subsec}

The 
notion of \emph{quantifier-free interdefinability}, which we define below,
is a variant of
the standard
notion of interdefinability from the
setting of $\aleph_0$-categorical theories (see, e.g.,
{\cite[\S1]{MR831437}}).
It provides a method for translating invariant measures concentrated on the extent of a given $\Pi_1$ theory 
to invariant measures
concentrated on the extent of a target $\Pi_1$ theory,
in a way that preserves the entropy function. 
We use this machinery to show, in Proposition~\ref{nonredundant-entropy-function}, that 
every entropy function of an invariant measure already occurs as the entropy function of
a non-redundant invariant measure for a relational language.

\emph{Throughout this subsection, $L_0$ and $L_1$ will be countable languages, sometimes with further restrictions.}

\begin{definition}
\label{interdefinability}
	Suppose $T_0$ is an $L_0$-theory and $T_1$ is an $L_1$-theory. A \defn{quantifier-free interdefinition} between 
$T_0$ and $T_1$
is a pair 
$\Psi=(\Psi_0, \Psi_1)$
of maps
\begin{eqnarray*}
	\Psi_0 &\colon &  \Lww(L_0) \to \Lww(L_1) \qquad \text{and}\\
	\Psi_1 & \colon & \Lww(L_1) \to \Lww(L_0)
\end{eqnarray*}
such that for $j \in \{0, 1\}$, the formula 
$\Psi_j(\eta)$ is quantifier-free whenever $\eta\in\Lww(L_{1-j})$ is quantifier-free, and
further,
\begin{eqnarray*}
T_{1-j}
& \vdash & \ \Psi_{j} \circ \Psi_{1-j} (\rho) \, \leftrightarrow \, \rho,\\
T_{1-j}
	& \vdash & \ \Psi_{j} (x=y) \, \leftrightarrow \, (x=y),\\
T_{1-j} & \vdash & \  \neg \Psi_j(\chi) \, \leftrightarrow \, \Psi_j(\neg \chi), \\
T_{1-j}
& \vdash & \  \Psi_j(\chi \And \varphi) \, \leftrightarrow \,
	\bigl(\Psi_j(\chi) \And \Psi_j(\varphi)\bigr), \qquad \text{and} \\
T_{1-j}
& \vdash& \  (\exists x) \Psi_j(\psi(x)) \, \leftrightarrow \,
\Psi_j\bigl((\exists x) \psi(x)\bigr)
\end{eqnarray*}
	for all  $\rho\in\Lww(L_{1-j})$ and
$\chi,\,\varphi,\,
\psi(x) 
\in \Lww(L_j)$,
and such that the free variables of $\Psi_j(\upsilon)$ are the same as those of $\upsilon$ for every $\upsilon\in \Lww(L_j)$.

	For $\fkn\in\Nats\cup\{\Nats\}$,
	the interdefinition $\Psi$ induces maps
	$\Psi^*_{j,\fkn}\colon \Extent{T_j}_{[\fkn]} \to \Extent{T_{1-j}}_{[\fkn]}$ for $j\in\{0,1\}$ satisfying,
	for any structure $\N\models T_j$, tuple $\m\in \fkn$, and formula $\varphi\in\Lww(L_{1-j})$,
	\[
		\Psi^*_{j,\fkn}(\N) \models \varphi(\m)
		\qquad \text{if and only if}\qquad
		\N \models \Psi_{1-j}(\varphi)(\m).
		\]
	It is immediate that each $\Psi^*_{j,\fkn}$ is a bijection.
\end{definition}

	Quantifier-free interdefinitions between $\Pi_1$ theories preserve entropy functions, as we now show.

\begin{lemma}
\label{Interdefinitions-preserve-entropy}
	Let $\Psi = (\Psi_0, \Psi_1)$ 
	be a quantifier-free interdefinition between a $\Pi_1$ $L_0$-theory 
	$T_0$ and a $\Pi_1$ $L_1$-theory $T_1$.
	Suppose $\mu$ is an invariant measure on $\Str_{\Lang_0}(\Nats)$ concentrated on $\Extent{T_0}$ and let $\nu$ be the pushforward of $\mu$ along $\Psi^*_{0,\Nats}$.
	Then 
	$\nu$ is an invariant measure on $\Str_{\Lang_1}(\Nats)$ concentrated on $\Extent{T_1}$, and $\Entropy{\nu} = \Entropy{\mu}$.
\end{lemma}

\begin{proof}
	It is immediate that $\nu$ is an invariant measure on $\Str_{\Lang_1}(\Nats)$ concentrated on $\Extent{T_1}$,
	by the definition of pushforward and the fact that  
$\Psi^*_{0,\Nats}$ is a bijection between $\Extent{T_0}$ and $\Extent{T_1}$.

	We will next show,
	for each $n\in\Nats$,
	that there is a measure-preserving bijection between the atoms 
	of $\mu_n$ and the atoms of $\nu_n$. 
By Definition~\ref{entropy-definition-finite-set},
	this will establish that $h(\mu_n) = h(\nu_n)$ for each $n\in\Nats$, and so
	$\Entropy{\mu} = \Entropy{\nu}$.

	Let $n\in\Nats$.
	If $\M\in\Extent{T_0}_\Nats$ then	
	$\M|_{[n]} \in \Extent{T_0}_{[n]}$, 
	because $\mu$ concentrates on $\Extent{T_0}_\Nats$ and $T_0$ is $\Pi_1$.
Hence the measure 
	$\mu_n$ concentrates on $\Extent{T_0}_{[n]}$.
	Similarly, $\nu_n$ concentrates on $\Extent{T_1}_{[n]}$.
	Therefore every atom of $\mu_n$ is in $\Extent{T_0}_{[n]}$
and every atom of $\nu_n$ is in $\Extent{T_1}_{[n]}$.

	Write $\nu^*_n$ for the pushforward of $\mu_n$ along $\Psi^*_{0,[n]}$. 
	It is clear that $\Psi^*_{0, [n]}$ is a measure-preserving bijection between the 
	atoms of $\mu_n$ and the atoms of $\nu^*_n$. 
	Let $\cB$ be an atom of $\nu_n$; 
	we have $\cB \in \Extent{T_1}_{[n]}$ because $\nu_n(\cB) > 0$.
	We will show that $\nu_n(\cB) = \nu^*_n(\cB)$.

	Let $\cA = \Psi_{1, [n]}^*(\cB)$.  
	For $\N\in\Extent{T_1}_\Nats$, 
		by Definition~\ref{interdefinability}
	we have
	$\N|_{[n]} = \cB$ if and only if $\Psi_{1, \Nats}^*(\N)|_{[n]} = \cA$.
		By this fact, the definitions of $\mu_n$, $\nu_n$, and pushforward, 
		the surjectivity of 
	$\Psi^*_{1, \Nats}$, and the fact that $\mu$ and $\nu$ are concentrated on
	$\Extent{T_0}_\Nats$ and $\Extent{T_1}_\Nats$ respectively,
		we have 
	\begin{eqnarray*}
\nu_n(\cB)
		& = & \nu(\{\N \in \Extent{T_1}_\Nats \st \N|_{[n]} = \cB\})\\
		& = & \nu(\{\N \in \Extent{T_1}_\Nats \st \Psi_{1, \Nats}^*(\N)|_{[n]} = \cA\})\\
		& = & \mu(\{\M  \in \Extent{T_0}_\Nats \st \M|_{[n]} = \cA\})\\
		& = & \mu_n(\cA).
	\end{eqnarray*}
	But $\mu_n(\cA) = \nu_n^*(\cB)$ by the definition of pushforward.
	Hence $\nu_n(\cB) = \nu_n^*(\cB)$, as desired.
\end{proof}

Non-redundant invariant measures are defined only for relational languages.
Hence towards showing,
in Proposition~\ref{nonredundant-entropy-function},
that every entropy function 
is the entropy function of a non-redundant invariant measure, we need a quantifier-free interdefinition 
between the empty theory
in an arbitrary 
language and a particular $\Pi_1$ theory in a relational language.
There is a standard such
interdefinition that
maps every function symbol 
to a relation symbol
representing the graph of the function,
but it entails an increase in arity, as each function symbol of arity $\ell$ is mapped to a relation symbol of arity $\ell+1$. 

While this standard interdefinition would suffice for our purposes 
in Proposition~\ref{nonredundant-entropy-function}, 
here we provide a more parsimonious interdefinition, to highlight a connection with 
what is known about invariant measures for languages containing function symbols.
Namely,
using results from \cite{AFPcompleteclassification}, 
in Lemmas 
\ref{AFP-consequence} 
and
\ref{atomic definition with functions} 
we
show how to
avoid increasing the arity of symbols,
by providing a quantifier-free interdefinition that
replaces each function symbol of arity $\ell$ with 
finitely many relation symbols, each of arity~$\ell$.

As we will see in Lemma~\ref{AFP-consequence}, 
every invariant measure is concentrated on structures in which
every function is a ``selector'', sometimes called a
``choice function'', i.e., a function for which the output is always one of the inputs.
For example, the only unary selector is the identity function. We consider constant symbols to be $0$-ary function symbols; observe that no constant is a selector.

\begin{definition}
	\label{Thfun-def}
	Let $f \in \Lang_0$ be a function symbol, and let $\ell$ be the arity of $f$. Define $\theta_f$ to be the sentence
\[
	(\forall x_0, \dots, x_{\ell-1})
\Big( \bigvee_{i \in [\ell]} f(x_0, \dots, x_{\ell-1}) = x_i \Bigl ),
\]
asserting that $f$ is a \emph{selector},
and define the $\Pi_1$ $\Lang_0$-theory 
	\[
	\Thfun(\Lang_0) \defas
		\{\theta_f \st f \in\Lang_0\text{ is a function symbol}\} .
		\]
\end{definition}

\begin{lemma}
\label{AFP-consequence}
If $\mu$ is an invariant measure on $\Str_{\Lang_0}(\Nats)$ then $\mu(\Extent{\Thfun(\Lang_0)}) = 1$. 
\end{lemma}
\begin{proof}
	Let $\nu$ be an ergodic invariant measure on $\Str_{\Lang_0}(\Nats)$.
By \cite[Lemma~2.4]{AFPcompleteclassification},
	\[
		\Th(\nu)\defas \{ \varphi \in \Lww(L) \st \nu(\Extent{\varphi}) = 1\},
		\]
		is a complete deductively-closed first-order theory. 
For a function symbol $f\in \Lang_0$, 
	if $\neg \theta_f \in \Th(\nu)$ then $\Th(\nu)$ has non-trivial definable closure, contradicting
	\cite[Proposition~6.1]{AFPcompleteclassification}, and so we must have $\theta_f \in \Th(\nu)$. 

		Therefore $\nu(\Extent{\Thfun(\Lang_0)}) = 1$.
	Because $\mu$ is a mixture of ergodic invariant measures $\nu$, we also have
	$\mu(\Extent{\Thfun(\Lang_0)}) = 1$.
\end{proof}

We now provide the desired quantifier-free interdefinition.

\begin{lemma}
\label{atomic definition with functions}
	Suppose $\Lang_1$ is the relational language 
consisting of
\begin{itemize}
\item each relation symbol in $\Lang_0$, along with

\item 
a relation symbol $E_{f, i}$
	for each function symbol $f \in \Lang_0$ and $i\in [\ell]$, where $\ell$ is the arity of $f$.
\end{itemize}

Let $T$ be the $\Pi_1$ $\Lang_1$-theory consisting of
the sentences
\[
(\forall x_0, \dots, x_{\ell-1}) \bigvee_{i \in [\ell]}E_{f, i}(x_0, \dots, x_{\ell-1}),
\]
and 
\[
(\forall x_0, \dots, x_{\ell-1}) \bigwedge_{i < j \in [\ell]}\neg \bigl(E_{f, i}(x_0, \dots, x_{\ell-1}) \And E_{f, j}(x_0, \dots, x_{\ell-1})\bigr),
\]
for each function symbol $f \in \Lang_0$, 
where $\ell$ is the arity of $f$.

	Then there is a quantifier-free interdefinition $(\Psi_0,\Psi_1)$
	between $\Thfun(\Lang_0)$ and $T$.
\end{lemma}
\begin{proof}
	For each relation symbol $R\in\Lang_0$, let $\Psi_1(R(x_0, \dots, x_{k-1}))$ be the formula $R(x_0, \dots, x_{k-1})$, where $k$ is the arity of $R$. For each function symbol $f\in \Lang_0$ and $i \in [\ell]$, let $\Psi_1(E_{f, i}(x_0, \dots, x_{\ell-1}))$ be the formula $f(x_0, \dots, x_{\ell-1}) = x_i$, where $\ell$ is the arity of $f$. It is easy to define the analogous map $\Psi_0$ and check that $(\Psi_0, \Psi_1)$ is a quantifier-free interdefinition
	between $\Thfun(\Lang_0)$ and $T$.
\end{proof}


Next we provide a quantifier-free interdefinition which, combined
with the previous results, 
will allow us to prove
	Proposition~\ref{nonredundant-entropy-function}.

\begin{lemma}
\label{Reduction to non-redundant}
	Suppose $\Lang_0$ is relational. Then there is a countable relational language $\Lang_1$
	and a quantifier-free interdefinition $(\Psi_0, \Psi_1)$ between
	the empty $\Lang_0$-theory and the $\Pi_1$ $\Lang_1$-theory $\Thnr(\Lang_1)$.
\end{lemma}

\begin{proof}
	For each relation symbol $R\in \Lang_0$ and
	equivalence relation $\Equiv$ on $[k]$,
	where $k\in\Nats$ is the arity of $R$, we define the following.
	Let $\ell$ be the number of $\Equiv$-equivalence classes.
	Let $f_\Equiv\colon [k] \to [k]$ send each $i\in[k]$ to the least element of its $\Equiv$-equivalence class,
	and let $\<y^\Equiv_0, \dots, y^\Equiv_{\ell-1}\>$ be the increasing enumeration of the image of $f_\Equiv$.
	Let $R_{\Equiv}$ be an $\ell$-ary relation symbol and let 
\[
	\phi_{\Equiv}(x_0, \dots, x_{k-1}) \defas R_{\Equiv}(x_{y_0^\Equiv}, \dots, x_{y_{\ell-1}^\Equiv}) \And \bigwedge_{\substack{i, j \in [k]\\ i \Equiv j}} x_i = x_j
	\And \bigwedge_{\substack{i, j \in [k]\\ \neg(i \Equiv j)}} x_{i}  \neq x_{j}
	.
\]

	Define $\Lang_1$ to be the collection of all such symbols $R_{\Equiv}$. 
Recall the $\Pi_1$ theory $\Thnr(\Lang_1)$, as defined in Definition~\ref{Thnr-def}.

	Define the map $\Psi_0$ on atomic $\Lang_0$-formulas by
\[
	\Psi_0\bigl(R(x_0, \dots, x_{k-1})\bigr) \defas 
	\bigvee_{\substack{\Equiv\text{~is an equivalence}\\\text{relation on }[k]} }
	\phi_{\Equiv}(x_0, \dots, x_{k-1}),
\]
for each $R\in\Lang_0$, where $k$ is the arity of $R$.
	Define the map $\Psi_1$ on atomic $\Lang_1$-formulas by
\[
	\Psi_1\bigl(R_\Equiv(x_0, \dots, x_{\ell-1})\bigr) \defas
	R(x_{f_\Equiv(0)}, \dots, x_{f_\Equiv(k-1)})
\]
for each $R_\Equiv\in\Lang_1$, where $\ell$ is the arity of $R_\Equiv$.
Extend
	$\Psi_0$ and $\Psi_1$ to all formulas in $\Lww(\Lang_0)$ and $\Lww(\Lang_1)$, respectively, in the natural way.

Observe that for any $\M\in \Str_{L_0}(\Nats)$,  the $L_1$-structure
$\Psi^*_{0,\Nats}(\M)$ is non-redundant, and conversely, every non-redundant structure in
$\Str_{L_1}(\Nats)$ is in the image of $\Psi^*_{0,\Nats}$.
One can check 
	that $(\Psi_0, \Psi_1)$ is a quantifier-free interdefinition between the empty $\Lang_0$-theory and $\Thnr(\Lang_1)$.
\end{proof}


We can now show that every entropy function is the entropy function of some non-redundant invariant measure.

\begin{proposition}
	\label{nonredundant-entropy-function}
	Let $\mu$ be an invariant measure on $\Str_{\Lang_0}(\Nats)$.
	Then there is a countable relational language $\Lang_1$ and a non-redundant invariant measure $\nu$ on 
	$\Str_{\Lang_1}(\Nats)$ such that 
	$\Entropy{\mu} = \Entropy{\nu}$.
	Further, if $L_0$ is finite and of maximum arity $k$,
	then so is $L_1$.
\end{proposition}
\begin{proof}
	Recall the $\Pi_1$ $L_0$-theory $\Thfun(\Lang_0)$ from Definition~\ref{Thfun-def}.
	First observe by 
	Lemma~\ref{AFP-consequence} that $\mu(\Extent{\Thfun(\Lang_0)}) = 1$.  
	By Lemma~\ref{atomic definition with functions} 
	there is a quantifier-free interdefinition $(\Psi_0, \Psi_1)$ between 
	$\Thfun(\Lang_0)$ and
	some $\Pi_1$ theory $T$ in a countable relational language $L'$. 
	When $L_0$ is finite and of maximum arity $k$, so is $L'$.

By Lemma~\ref{Reduction to non-redundant}
		there is a countable relational language $\Lang_1$, and a quantifier-free interdefinition $\Theta = (\Theta_0, \Theta_1)$
between the empty $L'$-theory and the $\Pi_1$ $\Lang_1$-theory $\Thnr$.
		It is easy to see from Definition~\ref{interdefinability} that $\Theta$ is also 
	a quantifier-free interdefinition between
	$T$ and its image $\Theta_0(T)$, which is a $\Pi_1$ $\Lang_1$-theory.
	Again, when $L'$ is finite and of maximum arity $k$, so is $L_1$.

	By two applications of Lemma~\ref{Interdefinitions-preserve-entropy}, there is an invariant measure $\nu$ on $\Str_{L_1}(\Nats)$ that is concentrated on
	$\Extent{{\Theta_0}(T)}$ and that satisfies $\Entropy{\mu} = \Entropy{\nu}$.
	Finally, $\nu$ is non-redundant because ${\Theta_0}(T)$ contains $\Thnr$.
\end{proof}

\section{Invariant measures arising from extended $L$-hypergraphons}
\label{section-finite}
In this section we study the growth of entropy functions of non-redundant invariant measures
for a relational language $L$ whose relation symbols all have the same arity $k \ge 1$.
By a 
variant of the 
  Aldous--Hoover--Kallenberg theorem, 
these invariant measures are precisely the ones that arise as
  the distribution of a
certain random $L$-structure $G(\Nats, W)$, where $W$ is 
  a type of measurable function called an \emph{extended $L$-hypergraphon}.
In Theorem~\ref{Entropy-from-hypergraphon}
  we express the growth of the entropy function of $G(\Nats, W)$
  in terms of $W$.
  This result
  generalizes a theorem 
of Aldous \cite[Chapter~15]{MR883646} 
 and Janson \cite[Theorem~D.5]{MR3043217},
 and our argument mirrors that of Janson.
	 As an immediate consequence of Theorem~\ref{Entropy-from-hypergraphon} and the Aldous--Hoover--Kallenberg theorem, we see
in Corollary~\ref{finite-relational-cor}
that when $L$ is finite, the entropy function of any non-redundant
invariant measure for $\Lang$ is
 $O(n^k)$.

\emph{For the rest of this section,
fix $k\ge1$
and let $\Lang$ be a countable relational language (possibly infinite) all of whose relations have arity $k$.}

	We will define an extended $L$-hypergraphon to 
be a probability kernel from $[0,1]^{\Powerset_{<k}(k)}$ 
to a space of quantifier-free $k$-types, satisfying a specific coherence condition.

\begin{definition}
	A \defn{complete non-redundant quantifier-free $\Lang$-type} with free variables $x_0, \dots, x_{k-1}$ is a maximal consistent collection of atomic formulas or negations of atomic formulas 
	containing $\{ x_i \neq x_j \st i < j < k\}$, and
	whose free variables are all contained in $\{x_0, \dots, x_{k-1}\}$.

	 Let $\qf_\Lang$ 
	be the space of complete non-redundant quantifier-free $\Lang$-types with free variables $x_0, \dots, x_{k-1}$, where a subbasic clopen set consists of complete non-redundant quantifier-free $\Lang$-types containing a given atomic formula or negated atomic formula all of whose free variables are contained in $\{x_0, \dots, x_{k-1}\}$.
\end{definition}

Note that every complete non-redundant quantifier-free $\Lang$-type in the above sense implies a complete  quantifier-free  $\Lwow(\Lang)$-type.

		In the next definition, 
we introduce the notion of an extended $\Lang$-hypergraphon.
As we describe below, this 
generalizes the standard notion of a hypergraphon \cite[\S23.3]{MR3012035}, which is a higher-arity version of a graphon.

		The definition of extended $L$-hypergraphon 
		involves a coherence condition, specified in terms of an action of $\Sym{k}$.
	In Definition~\ref{zeta-def}, we describe how an extended $\Lang$-hypergraphon 
	gives rise to 
	a distribution on $\Lang$-structures 
	by determining, for each $k$-tuple of elements, the distribution on its quantifier-free type.
The coherence condition 
ensures
	that the order in which the $k$-tuple is specified does not affect the resulting distribution on its quantifier-free type.

Consider the action of $\Sym{k}$ on $\qf_\Lang$ given by
\[
	\sigma \cdot u  = \{\varphi(x_{\sigma(i_0)}, \ldots, x_{\sigma(i_{\ell-1})}) \st \varphi(x_{i_0}, \ldots, x_{i_{\ell-1}})\in u 
	\textrm{~and~} \ell \le k\}
	\]
for $\sigma\in\Sym{k}$ and $u \in \qf_\Lang$.
Note that this action of $\Sym{k}$ extends to $\ProbMeas(\qf_\Lang)$ in the natural way, namely,
\[
	(\sigma \cdot \nu) (B)
= 	\nu(\{\sigma \cdot u \st u\in B\})
\]
for $\nu \in \ProbMeas(\qf_\Lang)$
	and Borel $B \subseteq \qf_\Lang$.

\begin{definition}
	\label{def-extended-hypergraphon}
An \defn{extended $\Lang$-hypergraphon} is a measurable map 
\[
W\colon[0,1]^{\Powerset_{<k}(k)} \to \ProbMeas(\qf_\Lang)
\]
	such that for any $\<x_F\>_{F\in \Powerset_{<k}(k)}\in [0,1]^{\Powerset_{<k}(k)}$ and 
	$\sigma \in \Sym{k}$,
	\[
		 W(\<x_{\sigma(F)}\>_{F\in \Powerset_{<k}(k)})
		 = 
		\sigma \cdot W(\<x_F\>_{F\in \Powerset_{<k}(k)}).
		\]
\end{definition}

The next technical lemma and definitions show how an
extended $L$-hyper\-graphon $W$ gives rise to a random non-redundant $L$-structure $G(\Nats, W)$ whose distribution is an invariant measure on $\Str_\Lang(\Nats)$.

The following is a special case of a standard result from probability theory about the randomization of a kernel.

\begin{lemma}[\textnormal{see {\cite[Lemma~3.22]{MR1876169}}}]
\label{Hypergraphon to AH representation}
	Let $W\colon[0,1]^{\Powerset_{<k}(k)} \to \ProbMeas(\qf_\Lang)$ be an extended $\Lang$-hypergraphon. There is a measurable function $W^*\colon[0,1]^{\Powerset_{<k}(k)}\times [0,1] \to \qf_\Lang$ such that whenever $\zeta$ is a uniform random variable in $[0,1]$, then for all $t \in [0,1]^{\Powerset_{<k}(k)}$, the random variable $W^*(t, \zeta)$ has distribution $W(t)$. 
\end{lemma}

Definition~\ref{hat-def}
introduces notation that we will use throughout this section and the next.

\begin{definition}
	\label{hat-def}
	Let $\fkn \in \Nats \cup \{\Nats\}$,
	and suppose $\fkn\geq k$ (when $\fkn \in \Nats$).
	Let $J\in \Powerset_k(\fkn)$,
	and
let $\tau_J\colon [k] \to J$ be the unique increasing bijection from $[k]$ to $J$. 
Define 
$\hat{X}_{J}$ to be the sequence 
$\<X_{\tau_J(F)}\>_{F \in \Powerset_{<k}(k)}$
		consisting of terms of the form $X_I$, where $I$ is from $\Powerset_{<k}(\fkn)$.
\end{definition}

	Recall that $L$ consists of relation symbols all of arity $k$, and so in order to define a non-redundant $L$-structure, it suffices to describe the complete quantifier-free type of every strictly increasing $k$-tuple in the structure.

\begin{definition}
	\label{zeta-def}
	Let $W$ be an extended $\Lang$-hypergraphon and $\fkn \in \Nats \cup \{\Nats\}$,
	and suppose $\fkn\geq k$ (when $\fkn \in \Nats$).

	\begin{itemize}
	\item 
		Define $M(\fkn, W)\colon [0, 1]^{\Powerset_{\leq k}(\fkn)} \to \Str_\Lang(\fkn)$ to be the map such that for all $\<x_D\>_{D\in \Powerset_{\leq k}(\fkn)} \in [0, 1]^{\Powerset_{\leq k}(\fkn)}$
		and every 
		$J\in\Powerset_k(\fkn)$, 
			the quan\-tifier-\linebreak{}free type of the tuple $\<\tau_J(0), \ldots, \tau_J(k-1)\>$ in
		the $L$-structure  \linebreak
		$M(\fkn, W)(\<x_D\>_{D \in \Powerset_{\le k}(\fkn)})$
		is
		$W^*(\hat{x}_{J}, x_J)$.
	\item 
		Let $\< \zeta_D \>_{D\in \Powerset_{\leq k}(\fkn)}$ be an i.i.d.\ sequence of uniform random \linebreak variables in $[0, 1]$. Define $G(\fkn, W)$ to be the random $\Lang$-structure $M(\fkn, W)(\< \zeta_D \>_{D\in \Powerset_{\leq k}(\fkn)})$.

	\item For $J \in \Powerset_k(\fkn)$,
	define the random variable $E^W_J \defas W^*(\hat{\zeta}_{J}, \zeta_J)$.
	\end{itemize}
\end{definition}

	In summary, $G(\fkn, W)$ is the random $L$-structure with underlying set $[\fkn]$
whose quantifier-free $k$-types are given by the random variables $E^W_J$ for $J \in \Powerset_k(\fkn)$.
It is easy to check that the distribution of $G(\fkn,W)$ does not depend on the specific choice of i.i.d.\ uniform $\< \zeta_D \>_{D\in \Powerset_{\leq k}(\fkn)}$ or on the specific function $W^*$ satisfying Lemma~\ref{Hypergraphon to AH representation}. 

Observe that for an extended $L$-hypergraphon $W$, the distribution of $G(\Nats, W)$ is a
probability measure on $\Str_\Lang(\Nats)$ that is invariant because \linebreak $\< \zeta_D \>_{D\in \Powerset_{\leq k}(\Nats)}$ is i.i.d., and is non-redundant because $W^*$ takes values in $\qf_\Lang$.
In fact, as 
Theorem~\ref{main-hypergraphon-theorem} asserts,
every non-redundant invariant measure on $\Str_\Lang(\Nats)$ arises from some $W$ in this way; this result
is a variant of the usual Aldous--Hoover--Kallenberg theorem, and
is a specialization of \cite[Theorem~2.37]{AckermanAutM}.

\begin{theorem}[{Aldous--Hoover--Kallenberg theorem}]
\label{main-hypergraphon-theorem}
	A non-redun\-dant probability measure $\mu$ on $\Str_\Lang(\Nats)$ 
is invariant under the action of $\Sym{\Nats}$ 
if and only if $\mu$ is the distribution of $G(\Nats, W)$ for some extended $\Lang$-hypergraphon $W$. 
\end{theorem}

	For more details on the usual Aldous--Hoover--Kallenberg theorem, which is stated in terms of exchangeable arrays, see \cite[Chapter~7]{MR2161313} and the historical notes to that chapter.
There is an essentially equivalent statement, in terms of hypergraphons, 
that 
provides a correspondence between hypergraphons (of arity $k$) and ergodic
invariant measures on the space of $k$-uniform graphs
	(see \cite{MR2463439} and \cite{MR2426176}).
The version that we have stated above in Theorem~\ref{main-hypergraphon-theorem}
provides an analogous correspondence between extended $L$-hypergraphons and 
non-redundant
invariant measures on 
$\Str_\Lang(\Nats)$. 
Note that,
unlike those invariant measures that arise from hypergraphons,
	the ones arising in Theorem~\ref{main-hypergraphon-theorem}
	can in general be mixtures of ergodic invariant measures, 
and can be concentrated on structures that 
include multiple relations (though finitely many, of the same arity) which need not be symmetric.
	
	In fact, we could have defined extended $\Lang$-hypergraphons, and then described the associated sampling procedure $G(\Nats, W)$ and stated the Aldous--Hoover--Kallenberg theorem, for the even more general case that does not stipulate non-redundance. However, the 
	invariant measure resulting from the sampling procedure would then be sensitive to measure $0$ changes in the extended $\Lang$-hypergraphon, 
	unlike the situation for standard hypergraphons or the usual Aldous--Hoover--Kallenberg representation, and so we have restricted our definitions and results to the non-redundant case.

	Using the machinery we have developed above, we can now prove the main theorems of this section.
By Theorem~\ref{main-hypergraphon-theorem}, 
in order to study the entropy function of a non-redundant invariant measure on $\Str_\Lang(\Nats)$,
we may 
ask for a suitable extended $L$-hypergraphon $W$ and 
then analyze the entropy function of $G(\Nats, W)$.
In Theorem~\ref{Entropy-from-hypergraphon}, we consider such a $G(\Nats, W)$, and express the leading term of its entropy function as a function of $W$.
As a consequence, in Corollary~\ref{finite-relational-cor}, when $L$ is finite we
obtain bounds on the entropy function of a non-redundant invariant measure for $L$.

The proof of 
Theorem~\ref{Entropy-from-hypergraphon}
closely follows that of \cite[Theorem~D.5]{MR3043217}.
\begin{theorem}
\label{Entropy-from-hypergraphon}
Fix an extended $\Lang$-hypergraphon $W$.
Suppose that \linebreak $C \defas \int \EntropyFin{W(\hat{z}_{[k]})}\ d\lambda(\hat{z}_{[k]})$
is finite.
Then
\[
\lim_{n \to \infty} \frac{\Entropy{G(\Nats, W)}(n)}{|\Powerset_{k}(n)|} = \int \EntropyFin{W(\hat{z}_{[k]})}\ d\lambda(\hat{z}_{[k]}).
\]
In particular, $\Entropy{G(\Nats, W)}(n) = C n^k + o(n^k)$ for some constant $C$.
\end{theorem}
\begin{proof}
	Observe that for all $n \in \Nats$, we have  $\Entropy{G(\Nats, W)}(n) = h(G(n, W))$.
	We first show a lower bound on $\EntropyFin{G(n, W)}$ for all $n\ge k$.  
	For $J \in \Powerset_k(n)$,
	the random variables $E^W_J$ 
	are independent conditioned on $\<\zeta_I\>_{I \in \Powerset_{<k}(n)}$
	and so, as the conditional entropy of conditionally independent random variables is additive,
we have
\begin{align*}
	\RandomEntropyFinbig{G(n, W) \givenbig \< \zeta_I\>_{I \in \Powerset_{<k}(n)}} \ &=\  \sum_{J \in \Powerset_k(n)} \RandomEntropyFinbig{E^W_J \givenbig \<\zeta_I\>_{I \in\Powerset_{<k}(n)}} \\
	\ &=\  \sum_{J \in \Powerset_k(n)} \RandomEntropyFinbig{\Exp\bigl(W(\hat{\zeta}_{J}) \givenbig  
	\<\zeta_I\>_{I \in \Powerset_{<k}(n)}\bigr)}\\
	\ &= \ \sum_{J \in \Powerset_k(n)} \RandomEntropyFin{W(\hat{\zeta}_{J})}
\end{align*}
	a.s.,
where the 
	last equality follows from the fact that every random variable in the sequence $\hat{\zeta}_{J}$  occurs within $\< \zeta_I\>_{I \in \Powerset_{<k}(n)}$.

By Lemma~\ref{cond-reduces-entropy}, 
	and then taking expectations of the first and last terms in the previous chain of equalities, we have
\begin{align*}
\label{Entropy sum of G(n W) given Bs}
\EntropyFinbig{G(n, W)} \ &\geq \ \Exp\bigl(\RandomEntropyFinbig{G(n, W) \givenbig 
	\<\zeta_I\>_{I \in \Powerset_{<k}(n)}}\bigr) \\
	\ &= \ \sum_{J \in \Powerset_k(n)} \Exp\bigl(\RandomEntropyFin{W(\hat{\zeta}_{J})}\bigr).
\end{align*}

	Because the distribution of the random variable
	$\hat{\zeta}_{J}$ is the same for all $J \in \Powerset_{k}(n)$, we have

\begin{align*}
	\sum_{J \in\Powerset_k(n)} \Exp\bigl(\RandomEntropyFin{W(\hat{\zeta}_{J})}\bigr)
\ &= \ |\Powerset_{k}(n)| \cdot  \Exp\bigl(\RandomEntropyFin{W(\hat{\zeta}_{[k]})}\bigr) \\
\ &= \ |\Powerset_{k}(n)| \cdot  \int \EntropyFin{W(\hat{z}_{[k]})}
\ d\lambda(\hat{z}_{[k]}).
\end{align*}
	This equation, together with the previous inequality, yields
\[
\frac{\EntropyFin{G(n, W)}}{|\Powerset_{k}(n)|} \ \geq \ \int \EntropyFin{W(\hat{z}_{[k]})}\ d\lambda(\hat{z}_{[k]}).
\]

Now we show an upper bound on $\EntropyFin{G(n, W)}$ for all $n\ge k$. Let $r \in \Nats$ be positive. 

	For $I \in \Powerset_{<k}(n)$, define $Y_I \defas \lfloor r \cdot \zeta_I \rfloor$, 
	so that
$Y_I = \ell$ precisely when $\frac{\ell}{r} \le \zeta_I < \frac{\ell + 1}{r}$. 
Then by Lemma~\ref{chainrule}, we have
\begin{align*}
\label{Entropy sum of G(n W) given Bs}
	\EntropyFinbig{G(n, W)} \ &\leq\  \EntropyFinbig{G(n, W), \<Y_I\>_{I \in \Powerset_{<k}(n)}} \\
	\ &=\ \EntropyFinbig{\<Y_I\>_{I \in\Powerset_{<k}(n)}} + \Exp\Bigl(\RandomEntropyFinbig{G(n, W) \givenbig \<Y_I\>_{I \in \Powerset_{<k}(n)}}\Bigr).
\end{align*}
	The first term in this last expression is straightforward to calculate; 
we have 
\[
	\EntropyFinbig{\<Y_I\>_{I \in \Powerset_{<k}(n)}} = \bigl |\Powerset_{<k}(n)\bigr | \cdot \log_2(r),
\]
as the $Y_I$ are uniformly distributed on $[r]$ and independent (because the $\zeta_I$ are i.i.d.\ uniform on $[0,1]$).

We next calculate the second term; we will show that
\begin{align*}
	\Exp\Bigl(\RandomEntropyFinbig{G(n, W) \given \<Y_I\>_{I \in \Powerset_{<k}(n)}}\Bigr) 
\ &=\  |\Powerset_{k}(n)|\cdot \int \EntropyFin{W_r(\hat{z}_{[k]})}\ d\lambda(\hat{z}_{[k]}).
\end{align*}

By (the standard conditional extension of) Lemma~\ref{joint-entropy-less-than-sum-entropy}, we have 
\begin{align*}
	\RandomEntropyFinbig{G(n, W) \givenbig \<Y_I\>_{I \in\Powerset_{<k}(n)}} 
	\ &\leq\  \sum_{J \in \Powerset_k(n)}\RandomEntropyFinbig{E^W_J \givenbig \<Y_I\>_{I \in \Powerset_{<k}(n)}} \\
\ &=\ \sum_{J \in \Powerset_k(n)}\RandomEntropyFinbig{E^W_J \givenbig \hat{Y}_{J}}
\end{align*}
a.s.,
	where the last equality follows from the fact that the only random variables in $\<\zeta_I\>_{I \in \Powerset_{<k}(n)}$ on which a given $E^W_J$ depends are those among $\hat{\zeta}_J$.

Given a function $\alpha\colon \Powerset_{<k}(k) \to [r]$,
define 
\begin{align*}
	w_r(\alpha)	\ \defas\  r^{|\Powerset_{<k}(k)|} \cdot \int W(\hat{z}_{[k]}) \cdot \prod_{F \in \Powerset_{<k}(k)} \Ind_{\bigl[\frac{\alpha(F)}{r},\frac{\alpha(F) + 1}{r}\bigr)}(z_{F})\ d \lambda(\hat{z}_{[k]}),
\end{align*}
where $\Ind_S$ denotes the characteristic function of a set $S$ (in this case, a half-open interval). 
Observe that 
\begin{align*}
w_r(\alpha) \ = \ \Exp\Big(W(\hat{\zeta}_{[k]}) \givenBig \bigwedge_{F \in \Powerset_{<k}(k)}Y_{F} = \alpha(F)\Big)
\end{align*}
a.s.
(We may form this conditional expectation 
because the set of
signed measures on $\qf_\Lang$ is a Banach space, and hence integration on 
this space
is well-defined.)

For $\hat{z}_{[k]} \in [0, 1]^{\Powerset_{<k}(k)}$,
define $\beta_{\hat z_{[k]}}\colon\Powerset_{<k}(k) \to [r]$ by $\beta_{\hat z_{[k]}}(F) = \lfloor r \cdot z_F \rfloor$ for $F \in \Powerset_{<k}(k)$, and
let 
$W_r(\hat{z}_{[k]}) \defas w_r(\beta_{\hat z_{[k]}})$.
Observe that while we have defined continuum-many instances of $\beta_{\hat z_{[k]}}$, 
they range over the merely finitely many functions from $\Powerset_{<k}(k)$ to $[r]$.
Further note that $W_r$ is a step function,
and
that as $r\to \infty$, the function  $W_r$ converges to $W$ pointwise a.e.
We can think of $W_r$ as the result of discretizing $W$ along blocks of width $1/r$.

Recall the function $\tau_J$ defined in Definition~\ref{hat-def}.
For any $J \in \Powerset_k(n)$, we have
\begin{eqnarray*}
	\Pr\Big(E^W_J 
	\hspace*{-2pt}
	\in \cdot \givenBig 
	\hspace*{-5pt}
	\bigwedge_{F \in \Powerset_{<k}(k)}
	\hspace*{-10pt}
	Y_{\tau_J(F)} = \beta_{\hat z_{[k]}}(F)\Big)
	\hspace*{-2pt}
	\ &=&\ 
	\Exp\Big(W(\hat{\zeta}_{J}) \givenBig 
	\hspace*{-5pt}
	\bigwedge_{F \in \Powerset_{<k}(k)}
	\hspace*{-10pt}
	Y_{\tau_J(F)} = \beta_{\hat z_{[k]}}(F)\Big ) \\
	\ &=&\ 
	\Exp\Big(W(\hat{\zeta}_{[k]}) \givenBig 
	\hspace*{-5pt}
	\bigwedge_{F \in \Powerset_{<k}(k)}
	\hspace*{-10pt}
	Y_{F} = \beta_{\hat z_{[k]}}(F)\Big ) \\
	\ &=&\  w_r(\beta_{\hat z_{[k]}})
\end{eqnarray*}
a.s.,
because 
$\tau_J$ is a bijection from $[k]$ to $J$ and
	the distribution of the random variable
	$\hat{\zeta}_{J}$ is the same as that of
	$\hat{\zeta}_{[k]}$. 
Therefore
	\[
		\EntropyFinBig{E^W_J \in \cdot \givenBig 
	\bigwedge_{F \in \Powerset_{<k}(k)}
	Y_{\tau_J(F)} = \beta_{\hat z_{[k]}}(F)}
	=
	\EntropyFin{w_r(\beta_{\hat z_{[k]}})}.
	\]
	Summing both sides over all possible
choices of 
$\beta_{\hat z_{[k]}}$,
	we have
	\[
		\sum_{\alpha\colon \Powerset_{<k}(k) \to [r]}\EntropyFinBig{E^W_J \in \cdot \givenBig 
	\bigwedge_{F \in \Powerset_{<k}(k)}
	Y_{\tau_J(F)} = \alpha(F)}
	=
	\hspace*{-7pt}
	\sum_{\alpha\colon\Powerset_{<k}(k) \to [r]}
	\hspace*{-4pt}
	\EntropyFin{w_r(\alpha)}).
	\]

As before,	
\[
	\RandomEntropyFinbig{E^W_J \givenbig \<Y_I\>_{I \in \Powerset_{<k}(n)}}
	= \RandomEntropyFinbig{E^W_J \givenbig \hat{Y}_J}.
	\]
	One can directly show that
\begin{align*}
	\Exp\Bigl(\RandomEntropyFinbig{E^W_J \givenbig \hat{Y}_J}\Bigr)
	= 
	r^{-|\Powerset_{<k}(k)|}\hspace*{-10pt}
	\sum_{\alpha\colon \Powerset_{<k}(k) \to [r]}
	\hspace*{-3pt}
	\EntropyFinBig{E^W_J \in \cdot \givenBig 
	\hspace*{-5pt}\bigwedge_{F \in \Powerset_{<k}(k)}
	\hspace*{-8pt}
	Y_{\tau_J(F)} = \alpha(F)}.
\end{align*}
Hence 
\begin{align*}
	\Exp\Bigl(\RandomEntropyFinbig{E^W_J \givenbig \<Y_I\>_{I \in \Powerset_{<k}(n)}}\Bigr) 
\ &=\ r^{-|\Powerset_{<k}(k)|} \cdot \sum_{\alpha\colon\Powerset_{<k}(k) \to [r]} \EntropyFin{w_r(\alpha)} \\
\ &=\  \int \EntropyFinbig{W_r(\hat{z}_{[k]})}\ d\lambda(\hat{z}_{[k]}).
\end{align*}
But once again by (an extension of) Lemma~\ref{joint-entropy-less-than-sum-entropy}, 
we have
\begin{align*}
	\Exp\Bigl(\RandomEntropyFinbig{G(n, W) \given \<Y_I\>_{I \in \Powerset_{<k}(n)}}\Bigr) 
	\ &\leq\ \sum_{J \in \Powerset_k(n)}\Exp\Bigl(\RandomEntropyFinbig{E^W_J \givenbig 
	\<Y_I\>_{I \in \Powerset_{<k}(n)}}\Bigr) \\
\ &=\  |\Powerset_{k}(n)|\cdot \int \EntropyFin{W_r(\hat{z}_{[k]})}\ d\lambda(\hat{z}_{[k]}).
\end{align*}

Putting together our two calculations,
we obtain 
\[
\EntropyFin{G(n, W)} \ \leq\  |\Powerset_{<k}(n)| \cdot \log_2(r) + |\Powerset_{k}(n)| \cdot  \int \EntropyFin{W_r(\hat{z}_{[k]})}\ d\lambda(\hat{z}_{[k]}).
\]
Hence for each positive $r$, we have
\[
\frac{\EntropyFin{G(n, W)}}{|\Powerset_{k}(n)|} \ \leq\ \frac{|\Powerset_{<k}(n)|}{|\Powerset_k(n)|} \cdot \log_2(r) +  \int \EntropyFin{W_r(\hat{z}_{[k]})}\ d\lambda(\hat{z}_{[k]}),
\]
and so
\[
\limsup_{n \to \infty}\frac{\EntropyFin{G(n, W)}}{|\Powerset_{k}(n)|}
\ \leq\  \int \EntropyFin{W_r(\hat{z}_{[k]})}\ d\lambda(\hat{z}_{[k]}).
\]
Letting $r \to \infty$, recall that $W_r \to  W$ a.e., and so
by the dominated convergence theorem, we have
\[
\limsup_{n \to \infty}\frac{\EntropyFin{G(n, W)}}{|\Powerset_{k}(n)|} 
\ \leq\ \int \EntropyFin{W(\hat{z}_{[k]})}\ d\lambda(\hat{z}_{[k]}).
\]

Combining our lower and upper bounds, we have
\[
\lim_{n \to \infty} \frac{\EntropyFin{G(\Nats, W)}}{|\Powerset_{k}(n)|} 
= \int \EntropyFin{W(\hat{z}_{[k]})}\ d\lambda(\hat{z}_{[k]}),
\]
as desired.

Finally, because $\lim_{n \to \infty} \frac{|\Powerset_k(n)|}{n^k} = 1$, we have
\[\Entropy{G(\Nats, W)}(n) = 
Cn^k + o(n^k),\] where
$C = \int \EntropyFin{W(\hat{z}_{[k]     })}\ d\lambda(\hat{z}_{[k]})$. 
\end{proof}

As a corollary, when $L$ is finite, we obtain a bound on the growth of entropy functions of non-redundant invariant measures for $L$.

\begin{corollary}
	\label{finite-relational-cor}
Suppose $L$ is finite, and
       let $\mu$ be a non-redundant invariant measure on
	$\Str_{\Lang}(\Nats)$.
	Then $\Entropy{\mu}(n) = Cn^k + o(n^k)$ for some constant $C$. 
	In particular, $\Entropy{\mu}(n) = O(n^k)$.
\end{corollary}
\begin{proof}
	By Theorem~\ref{main-hypergraphon-theorem}
	there is some extended $L$-hypergraphon $W$ such that
	$\mu$ is the distribution of the random $L$-structure $G(\Nats, W)$.
Because $L$ is finite, $W$ is bounded, and so
$\int \EntropyFin{W(\hat{z}_{[k]})}\ d\lambda(\hat{z}_{[k]})$
is finite.
       Hence the entropy function $\Entropy{\mu}$ is of the desired form
	by Theorem~\ref{Entropy-from-hypergraphon}.
\end{proof}

\section{Invariant measures sampled from a Borel hypergraph}
\label{randomfree-sec}

We have seen in Corollary~\ref{finite-relational-cor}
that for a finite relational language $\Lang$ all of whose relation symbols have the same arity $k \ge 1$,
the entropy function of an invariant measure on $\Str_{\Lang}(\Nats)$ is of the form $Cn^k + o(n^k)$, where 
$C$ is a constant depending on the invariant measure. 
In this section we consider the situation where $C=0$.

For $k=1$,
consider an
invariant measure $\mu$ on $\Str_{\Lang}$, 
and let $W$ be an extended $L$-hypergraphon such that $\mu$ is the distribution of $G(\Nats, W)$.
	By Theorem~\ref{Entropy-from-hypergraphon}, we have 
$\Entropy{\mu}(n) = C n + o(n)$,
where $C = \int \EntropyFin{W(z_{\emptyset})}\ d\lambda(z_{\emptyset})$.
	Suppose $C=0$. Then $h(W(z_\emptyset)) = 0$ for a.e.\ $z_\emptyset$,
	and so $W(z_\emptyset)$ is a point mass a.e.
	Hence $G(\Nats, W)$ is a random $L$-structure where every element of $\Nats$ has the same quantifier-free $1$-type, a.s.
	Therefore $\mu$ is a mixture of finitely many point masses,
	and $\Entropy{\mu}(n)$ is a constant that does not depend on $n$.
In summary, for $k=1$,
the only possible entropy functions of sublinear growth are the constant functions.

Theorem~\ref{Random-free has arbitrary growth rates},
the main result of this section,
states that in contrast, for $k>1$,
Corollary~\ref{finite-relational-cor}
is tight
in the sense that
	for any given function $\gamma$ that is $o(n^k)$, 
there is some non-redundant invariant measure whose entropy function 
is $o(n^k)$ but
grows faster than $\gamma$.
As discussed in the introduction, 
this result
is a generalization of
\cite[Theorem~1.1]{MR3073488}, and the arguments in this section 
closely follow their proof.

\emph{For the rest of this section, fix $k \ge 2$
and let $\Lang$
be the language of $k$-uniform hypergraphs, i.e., 
$\Lang = \{E\}$ 
where $E$ is a $k$-ary relation symbol.}

A \emph{$k$-uniform hypergraph}
is a non-redundant $\Lang$-structure satisfying 
\[
\bigwedge_{\sigma\in \Sym{k}} (\forall x_0, \dots, x_{k-1})\ \big(E(x_0, \dots, x_{k-1}) \leftrightarrow E(x_{\sigma(0)}, \dots, x_{\sigma(k-1)})\big),
\]
and we call the instantiation of $E$ its \emph{edge set}. 

By a \emph{Borel hypergraph}, we mean a $k$-uniform hypergraph $\cM$ whose underlying set is $[0,1]$ and such that for any atomic formula $\varphi$ in the language of hypergraphs, 
the set $\{\a \in \cM\st \cM \models \varphi(\a)\}$ of realizations of $\varphi$ in $\cM$ is Borel. 

Any extended $L$-hypergraphon $W$ \emph{yields} a non-redundant invariant measure, namely the distribution of the random $L$-structure $G(\Nats, W)$.
We say that 
	$W$ \emph{induces a Borel hypergraph} when this invariant measure can be obtained by sampling a random subhypergraph of some Borel hypergraph, as we make precise in Definition~\ref{def-induces-a-borel-hypergraph}.
We will see, in
Lemma~\ref{random-free-has-leading-entropy-zero},
that when $W$ induces a Borel hypergraph, it yields an invariant measure whose entropy function is $o(n^k)$.
The main construction of this section,
in Theorem~\ref{Random-free has arbitrary growth rates},
builds non-redundant invariant measures whose entropy functions have arbitrarily high growth within $o(n^k)$
by sampling from certain extended $L$-hypergraphons that induce Borel hypergraphs.

Observe that there are only two non-redundant quantifier-free $k$-types in $L$ that are consistent with the theory of $k$-uniform hypergraphs. 
Let $u_\top, u_\bot \in \qf_\Lang$ be
the unique non-redundant quantifier-free types
containing 
\[
	\bigwedge_{\sigma \in \Sym{k}} E(x_{\sigma(0)}, \dots, x_{\sigma(k-1)}) 
	\]
	and 
\[
	\bigwedge_{\sigma \in \Sym{k}} \neg E(x_{\sigma(0)}, \dots, x_{\sigma(k-1)}),
	\]
respectively, and let $\delta_\top, \delta_\bot \in \ProbMeas(\qf_\Lang)$ be the respective point masses concentrated on them.

Recall from Definition~\ref{hat-def}
	that $\hat{x}_{[k]}$ denotes the tuple of variables $\<x_I\>_{I \in \Powerset_{<k}(k)}$.

\begin{definition}
	\label{def-induces-a-borel-hypergraph}
\quad We say that an extended $\Lang$-hypergraphon \linebreak
	$W\colon[0,1]^{\Powerset_{<k}(k)} \to \ProbMeas(\qf_\Lang)$ \defn{induces a Borel hypergraph} 
if 
\begin{itemize}
\item for a.e.\ pair of sequences $\hat{x}_{[k]}$, $\hat{y}_{[k]}$ of elements of $[0,1]$ with $x_{\{i\}} = y_{\{i\}}$ for all $i\in[k]$, we have 
	$W(\hat{x}_{[k]}) = W(\hat{y}_{[k]})$, and 

\item for a.e.\ sequence $\hat{x}_{[k]}$ of elements of $[0,1]$, the distribution 
	$W(\hat{x}_{[k]})$ is either $\delta_\top$ or $\delta_\bot$. 
\end{itemize}
\end{definition}

It follows that an extended $\Lang$-hypergraphon $W$ induces a Borel hypergraph precisely when there is a Borel hypergraph $\cB$
such that 
$W(\hat{x}_{[k]})$ is a point mass concentrated on the 
quantifier-free type of $\<x_{\{0\}}, \dots, x_{\{k-1\}}\>$ in $\cB$.
In this case we say that $W$ \emph{induces the Borel hypergraph} $\cB$.

The notion of an extended $L$-hypergraphon inducing a Borel hypergraph is closely related, in the case $k=2$, to that of a graphon being random-free.
A \emph{graphon} is a symmetric Borel function from $[0,1]^2$ to $[0,1]$; as described in 
\cite{MR2815610} and \cite{MR3043217},
it is called
\emph{random-free} 
when it is $\{0,1\}$-valued a.e. 
Every graphon gives rise
to a random undirected graph on $\Nats$ whose distribution is an invariant measure.
For $k=2$,
an extended $L$-hypergraphon $W$ that yields
an invariant measure concentrated on undirected graphs 
can be expressed as a mixture of invariant measures, each obtained via a graphon.
In the case where such a $W$ corresponds to a single graphon, $W$ induces a Borel hypergraph precisely when the corresponding graphon is random-free.

	The notion of a random-free graphon also essentially appeared, in the context of separate exchangeability, in work of Aldous in \cite[Proposition~3.6]{MR637937} and \cite[(14.15) and p.~133]{MR883646},
and Diaconis--Freedman \cite[(4.10)]{MR640207}.
Further, Kallenberg \cite{MR1702867} describes, for all $k\ge2$, the similar notion of a \emph{simple array}; this corresponds to our notion of inducing a Borel hypergraph, in the case where the distribution of the simple array is ergodic.
Random-free graphons arise as well 
in \cite{MR2724668} and \cite[\S6.1]{MR3515800},
which consider invariant measures that are concentrated on a given orbit of the logic action.

		Theorem~\ref{Entropy-from-hypergraphon} implies that any extended $L$-hypergraphon $W$ yields
	an invariant measure whose
entropy function is $O(n^k)$.
Observe that there are extended $L$-hypergraphons achieving this upper bound, i.e., that yield an invariant measure whose entropy function is $\Omega(n^k)$.
For example, one can directly calculate that the \ER\ extended $L$-hypergraphon given by the constant function
\[
	W_\mathrm{ER}(\hat{x}_{[k]}) = \mathrm{Uniform}(\{u_{\top}, u_\bot\})
\]
satisfies
$\Entropy{G(\Nats, W_\mathrm{ER})}(n) = {n\choose k}$.

	However, we now show that this growth rate cannot be achieved for an extended $L$-hypergraphon that induces a Borel hypergraph.

\begin{lemma}
	\label{random-free-has-leading-entropy-zero}
	Let $W$ be an extended $L$-hypergraphon,
	and suppose $W$ induces a Borel hypergraph. 
	Then $\Entropy{G(\Nats, W)}(n) = o(n^k)$.
\end{lemma}
\begin{proof}
	Because $W$ induces a Borel hypergraph, it takes the value
	$\delta_\top$ or $\delta_\bot$ a.e.
	But 
	$\EntropyFin{\delta_{\top}} = \EntropyFin{\delta_{\bot}} = 0$,
	and so
$\int \EntropyFin{W(\hat{z}_{[k]})}\ d\lambda(\hat{z}_{[k]}) =0$.
Hence by Theorem~\ref{Entropy-from-hypergraphon}, we have
	$\Entropy{G(\Nats, W)}(n) = o(n^k)$.
\end{proof}

As noted previously,
Aldous \cite{MR883646}
and Janson \cite{MR3043217}
have versions of Theorem~\ref{Entropy-from-hypergraphon} for $k=2$. 
They also observe  
that their respective results immediately imply that a random-free graphon yields an invariant measure
with entropy function that is $o(n^2)$; their proofs are similar to that of
	Lemma~\ref{random-free-has-leading-entropy-zero}.
Their setting involves working with graphons,
which yield ergodic invariant measures.
Under the restriction of ergodicity,
it is easily seen that the converse of Lemma~\ref{random-free-has-leading-entropy-zero} holds for $k=2$, 
as 
noted by Janson \cite[Theorem~10.16]{MR3043217}.

In contrast, the converse of Lemma~\ref{random-free-has-leading-entropy-zero} itself does not hold, as
our notion of extended $L$-hypergraphon allows for ones that yield non-ergodic invariant measures.
For example, consider
the extended $L$-hypergraphon
\[
	\Wsqsq  (\hat{x}_{[k]}) = 
\begin{cases}
\delta_{\top}& \text{ if } x_\emptyset < \frac12,\\
\delta_{\bot}& \text{ otherwise}.
\end{cases}
	\]
	Define the extended $L$-hypergraphons
	$\Wbsq  (\hat{x}_{[k]}) = \delta_{\top}$ and
	$\Wwsq  (\hat{x}_{[k]}) = \delta_{\bot}$,
	which each induce a Borel hypergraph.
The random hypergraph $G(\Nats, \Wsqsq)$ is the complete hypergraph or the empty hypergraph, each with probability $\frac12$, and so its distribution
is a non-trivial mixture of the distributions of $G(\Nats, \Wbsq)$ and $G(\Nats, \Wwsq)$,
hence
		a non-ergodic invariant measure.
The extended $L$-hypergraphon $\Wsqsq$ does not induce a Borel hypergraph as it depends on the variable $x_\emptyset$, yet the entropy function of $G(\Nats, \Wsqsq)$
is $o(n^k)$. 

But in fact, for $k\ge3$, there is a more interesting obstruction to a converse
of Lemma~\ref{random-free-has-leading-entropy-zero},
even among extended $L$-hypergraphons that yield ergodic invariant measures.
The following example for $k=3$ (which is easily generalized to larger values of $k$)
	makes fundamental use of the variables indexed by pairs from $[3]$, and yet also yields an invariant measure whose entropy function is $o(n^3)$:
\[
	W_{\!\triangle}(\hat{x}_{[3]}) = 
\begin{cases}
	\delta_{\top}& \text{ if } x_{\{0,1\}}< \frac12 \text{ and }
	x_{\{0,2\}}< \frac12 \text{ and }
	x_{\{1,2\}}< \frac12,
	\\
\delta_{\bot}& \text{ otherwise}.
\end{cases}
\]
The random hypergraph $G(\Nats,W_{\!\triangle})$ can be thought of as first building a ``virtual'' \ER\ graph with independent $2$-edge probabilities $\frac12$, then adding a $3$-edge for each triangle existing in the graph, and then throwing away the virtual $2$-edges.
(For more about this example, see \cite[p.~92]{MR2426176} and \cite[Example~23.11]{MR3012035}.)

In Lemma~\ref{random-free-has-leading-entropy-zero},
we established
that any extended $L$-hypergraphon $W$ that induces a Borel hypergraph 
is such that
the entropy function of $G(\Nats,W)$ is $o(n^k)$.
We now proceed to show that there are such $W$ for which
the growth of
$\Entropy{G(\Nats,W)}(n)$
is arbitrarily close to $n^k$.


We first define a kind of ``blow up''
that creates an extended $L$-hypergraphon from a countably infinite $L$-structure. We will use this
notion in Lemma~\ref{Bound on conditional entropy}.
	Blow ups are a standard technique for expanding a countable structure into 
a continuum-sized structure that has a positive-measure worth of ``copies'' of each point from the original.
(See, e.g., the use of step function graphons in \cite{MR3012035}.)

\begin{definition}
	\label{piblowup-def}
		Let $\cM \in \Str_{\Lang}(\Nats)$
	and let $\pi\colon [0,1]\to \Nats$ be a Borel map such that $\lambda(\pi^{-1}(i)) > 0$ for all $i\in\Nats$.
		The \defn{$\pi$-blow up of $\cM$} is defined to be the extended $L$-hypergraphon
$W\colon[0,1]^{\Powerset_{<k}(k)} \to \{\delta_{\top}, \delta_{\bot}\}$ 
given by
\[
	W(\hat{x}_{[k]}) =
\begin{cases}
\delta_{\top}& \text{ if } 
			\cM \models E\bigl(\pi(x_{\{0\}}), \pi(x_{\{1\}}), \dots, \pi(x_{\{k-1\}})\bigr),\\
\delta_{\bot}& \text{ otherwise}.
\end{cases}
\]
\end{definition}

Observe that such a $W$ induces a Borel hypergraph, as on every input it outputs either the value $\delta_\top$ or $\delta_\bot$,
and it depends only on the variables
			$x_{\{0\}}, x_{\{1\}}, \dots, x_{\{k-1\}}$.
		In particular, $W$ induces the Borel hypergraph $\cB$ with underlying set $[0,1]$ given by:
\[
		\cB \models E(x_0, x_1,\dots, x_{k-1})
		\]
if and only if
		\[
			\cM \models E\bigl(\pi(x_0), \pi(x_1), \dots, \pi(x_{k-1})\bigr).
		\]

Hence we can think of
$\pi^{-1}$ as a Borel partition of the unit interval into positive measure pieces,
	such that each
	element of $\cM$ 
	is ``blown up'' 
	into a piece of the partition in $\cB$, and each piece of the partition arises in this way.

Next, we proceed to construct an extended $L$-hypergraphon $W$ that yields an invariant measure whose entropy function has the desired growth. 
	Analogously to Hatami--Norine \cite{MR3073488}, 
we let $W$ be the blow up of a particular countably infinite structure, the \emph{Rado $k$-hypergraph}, a well-known generalization of the \emph{transversal-uniform graph} used in \cite{MR3073488}.
	The Rado $k$-hypergraph is the countable homogeneous-universal $k$-uniform hypergraph,
	namely,
the unique (up to isomorphism) countable $k$-uniform hypergraph satisfying 
the so-called ``Alice's restaurant'' axioms. These axioms state that for any 
possible way of extending a finite induced subhypergraph of the Rado $k$-hypergraph
by one vertex to obtain a $k$-uniform hypergraph, there is some element of the Rado $k$-hypergraph
that realizes this extension.

We now describe an inductive construction of an instantiation $\Rado_k$ of the Rado $k$-hypergraph, with underlying set $\Nats$.
	At each stage $\ell\in\Nats$ we define a finite set 
	$\Alice_\ell$
	of new vertices,
	which we call \emph{generation} $\ell$,
	and build a $k$-uniform hypergraph $G_\ell$ with underlying set $V_\ell \defas \bigcup_{j\le \ell} \Alice_j$.

	Stage $0$: Let $A_0 \defas \{0\}$ consist of a single vertex, and let $G_0$ be the 
	empty hypergraph with vertex set $V_0 = \Alice_0$.

		Stage $\ell>0$: 
Let $\Alice_\ell$ 
consist of one new vertex $a_X$ for each subset $X$ of unordered $(k-1)$-tuples from $V_{\ell-1}$,
with the elements of $\Alice_\ell$
	chosen to be the consecutive least elements of $\Nats$ not yet used.
Let $G_\ell$ be the $k$-uniform hypergraph on vertex set $V_\ell = V_{\ell-1} \cup \Alice_\ell$ 
whose edges are those of $G_{\ell-1}$ along with, for each such $X$ and every unordered tuple $d\in X$, an edge consisting of $a_X$ and the $k-1$ vertices in $d$.

Define $\Rado_k$ to be the union of the hypergraphs $G_\ell$, i.e., the hypergraph with vertex set $\bigcup_{\ell\in\Nats} V_\ell$ and edge set $\bigcup_{\ell\in\Nats} E^{G_\ell}$.
	Observe that $\Rado_k$ is a $k$-uniform hypergraph with underlying set $\Nats$ that
	satisfies
the Alice's restaurant axioms: Given 
a finite induced subhypergraph $D$ of $\Rado_k$,
all one-vertex extensions of $D$ are realized in  
stage $\ell+1$, for any $\ell$ such that $V_\ell$ contains the 
vertices of $D$.

The following lemma, which we will use
in the proof of Theorem~\ref{Random-free has arbitrary growth rates},
provides a lower bound on the entropy function of a random hypergraph sampled from a blow up of $\Rado_k$ in the case where elements of $\Rado_k$ belonging to the same generation get blown up to sets of equal measure.
Both the statement and proof of the lemma
	are directly analogous to those of \cite[Lemma~2.1]{MR3073488}.

Recall from Definition~\ref{zeta-def} that
$\<\zeta_D\>_{D\in \Powerset_{\le k}(\Nats)}$ is the collection of i.i.d.\ uniform random variables in $[0,1]$ in terms of which the random $L$-structure $G(\Nats,W)$ is defined.

\begin{lemma}
\label{Bound on conditional entropy}
	Let $\pi\colon [0,1]\to \Nats$ be a Borel map such that $\lambda(\pi^{-1}(i)) > 0$ for all $i\in\Nats$, and
	let $W$ be a $\pi$-blow up of $\Rado_k$.
	Suppose that $\lambda(\pi^{-1}(a)) = \lambda(\pi^{-1}(b))$ for all $\ell\in\Nats$ and $a,b\in \Alice_\ell$.
Then for all $n \in \Nats$ and $\rho\colon[n] \to \Nats$, we have 
\[
	\RandomEntropyFinBig{G(n, W)
		\givenBig \bigwedge_{j \in [n]}  \pi(\zeta_{\{j\}}) \in \Alice_{\rho(j)}} 
	\ \geq \ {\bigl|\rho([n])\bigr| \choose {k}},
\]
a.s.
\end{lemma}
\begin{proof}
Let $S\subseteq [n]$ be maximal such that $\rho$ is injective on $S$,
	and write
	$G(S,W)$ to denote the random induced substructure of $G(n, W)$ with underlying set $S$.
Then	
\begin{align*}
	\RandomEntropyFinBig{
		G(n, W)
		\givenBig  \bigwedge_{j \in [n]}  \pi(\zeta_{\{j\}}) \in \Alice_{\rho(j)}} 
	\ \geq\  \RandomEntropyFinBig{
		G(S, W)
		\givenBig  \bigwedge_{j \in S}  \pi(\zeta_{\{j\}}) \in \Alice_{\rho(j)}}.
\end{align*}

	Consider
	the random measure
	\begin{align*}
		\label{condprob-uniform}
\tag{$\dagger$}
	\Exp\Big(
		G(S, W)
		\givenBig \bigwedge_{j \in S}  \pi(\zeta_{\{j\}}) \in \Alice_{\rho(j)}\Big).
	\end{align*}
	Since $W$ is a blow up of $\Rado_k$, 
	the random hypergraph $G(S,W)$ is a $|S|$-element sample with replacement from $\Rado_k$, with the vertices relabeled by $S$.
	By the injectivity of $\rho$, the condition in \eqref{condprob-uniform} constrains the elements of $G(S,W)$ to be obtained from distinct generations
	of $\Rado_k$, which implies that the random distribution on hypergraphs given by
		\eqref{condprob-uniform}
	is actually sampled from $\Rado_k$ without replacement a.s.

	Further, each $\zeta_{\{j\}}$ is uniform, and $\pi^{-1}$ assigns sets of equal measure to vertices in $\Rado_k$ of the same generation. So
		\eqref{condprob-uniform}
	is a.s.\ the distribution of the random hypergraph $Q$ with underlying set $S$ obtained by, for each $j\in S$, uniformly selecting a vertex of $\Rado_k$ from among those 
	in generation $\rho(j)$, and taking the edges induced from $\Rado_k$.

	Let $\ell \in \Nats$, and consider a subset $U\subseteq V_\ell$.
	Write $T$ for the set of unordered $(k-1)$-tuples from $U$.
	Suppose $\ell' > \ell$.
	For every
	subset $X \subseteq T$,
	exactly a $2^{-|T|}$-fraction of the vertices in
	$A_{\ell'}$ form an edge with every $(k-1)$-tuple in $X$ and with no $(k-1)$-tuple in $T \setminus X$.
	For a vertex $v$ selected uniformly at random from $A_{\ell'}$, let $G_{U, v}$ be
	the 
	(not necessarily induced) random
	subhypergraph of $G_{\ell'}$ 
	that has vertex set
	$U \cup \{v\}$ and whose edges are those in $G_{\ell'}$ that consist of $v$ along with
	$k-1$ vertices from $U$.
	Then $G_{U,v}$
	is equally likely to be 
	any of the hypergraphs with vertex set 
	$U \cup \{v\}$ whose edges all include $v$.

	For each $j\in S$, define $S_j \defas \{i \in S \st \rho(i) < \rho(j)\}$, and
	let $Q_j$ be the (not necessarily induced) random subhypergraph of $Q$
	that has vertex set $S_j \cup \{j\}$ and whose edges are those in $Q$ that consist of $j$ along with $k-1$ vertices from $S_j$.
	Then for each $j\in S$, 
	the random hypergraph $Q_j$
	is equally likely to be any given hypergraph $F_j$ with vertex set $S_j \cup \{j\}$ all of whose edges include $j$. 
	Further, for
	any choice of hypergraphs $\{F_j\st j\in S\}$ 
	as above,
	the events 
	$Q_j = F_j$ for $j\in S$
	are independent because the random variables $\zeta_{\{j\}}$ are independent.

	Now, for any hypergraph $F$ with underlying set $S$, we can write 
	$F$ as the union,
	over $j\in S$, of the
	subhypergraph of $F$
	that has vertex set $S_j \cup \{j\}$ and whose edges are those in $F$ consisting of $j$ along with $k-1$ vertices from $S_j$.
	We therefore see that $Q$ is equally likely to be any
	hypergraph on vertex set $S$.

	In summary, the 
	random distribution \eqref{condprob-uniform}
	is a.s.\ the uniform measure on
	$k$-uniform hypergraphs with underlying set $S$.
Hence	
	\[
	\RandomEntropyFinBig{
		G(S, W)
		 \givenBig  \bigwedge_{j \in S}  \pi(\zeta_{\{j\}}) \in \Alice_{\rho(j)}} \ =\  {{|S|}\choose {k}}
\]
a.s.,
establishing the lemma.
\end{proof}

We now prove the main result of this section, 
which asserts that the
entropy function of an invariant measure 
can have arbitrarily large growth rate within $o(n^k)$.
This
result is a higher-arity version of 
\cite[Theorem~1.1]{MR3073488},
and its proof proceeds via the same steps. We include the proof here (with appropriately modified parameters and notation) for completeness.

\begin{theorem}
\label{Random-free has arbitrary growth rates}
Suppose $\gamma\colon \Nats \to [0,1]$ is a function such that $\lim_{n\to \infty} \gamma(n) = 0$. 
	Then there is an extended $L$-hypergraphon $W$ that induces a Borel hypergraph and is such that 
	$\Entropy{G(\Nats, W)} (n) = o(n^k)$ and 
	$\Entropy{G(\Nats, W)} (n) = \Omega(\gamma(n) \cdot n^k)$.
\end{theorem}
\begin{proof}
	We will define a Borel map $\pi\colon [0,1]\to \Nats$ 
	satisfying $\lambda(\pi^{-1}(i)) > 0$ for all $i\in\Nats$  
	in such a way that
	the $\pi$-blow up of $\Rado_k$, which we denote by $W$,
	has the desired properties.

	By the observation that follows Definition~\ref{piblowup-def}, if
$\lambda(\pi^{-1}(i)) > 0$ for all $i\in\Nats$, then
the $\pi$-blow up
	$W$ induces a Borel hypergraph, and so 
	by Lemma~\ref{random-free-has-leading-entropy-zero}, we have
	$\Entropy{G(\Nats, W)} (n) = o(n^k)$.
	Hence it suffices to construct $\pi$ 
	satisfying $\lambda(\pi^{-1}(i)) > 0$ for all $i\in\Nats$ in such a way that 
	$\Entropy{G(\Nats, W)}(n) = $ \linebreak $\Omega(\gamma(n) \cdot n^k)$.

	For positive $r\in\Nats$, define
\[
	g_r \defas \max\bigl\{\{2^{r+3} k \} \cup \{n \in\Nats \st \gamma(n) > 2^{-(r+1)k-3k-1} k^{-k} \}\bigr\}.
	\]
	Note that $\lim_{n\to \infty} \gamma(n) = 0$ and so for each $r$ there are only finitely many $n$ such that 
	$\gamma(n) > 2^{-(r+1)k-3k-1} k^{-k}$;
hence $g_r$ is well-defined. 

	Observe that
	for all 
	$n\ge g_1 + 1$, the inequalities
\begin{align*}
	\label{ellstar}
\tag{$\star$}
	n > 2^{r+2} k
\qquad
\text{and}\qquad
\gamma(n) \le 2^{-r k-3k-1} k^{-k}
\end{align*}
hold when $r = 1$.
The remainder of the proof establishes that 
for all such $n$, we have
$\Entropy{G(\Nats, W)}(n) \ \geq \ n^k \cdot \gamma(n)$.

Fix $n\ge g_1+1$. We have seen that there is at least one
	$r$ satisfying \eqref{ellstar}; on the other hand,
	there are only finitely many 
	choices of $r$ for which \eqref{ellstar} holds.
Let $q$ be the largest such $r$.
Then either $n < 2^{q+3} k$ 
or \linebreak $\gamma(n) > 2^{-(q+1)k-3k-1} k^{-k}$, and so $n \le g_q$ by the definition of $g_q$.

For each $r\ge1$, define 
	\[\textstyle
		\Gamma_r \defas \bigl\{\ell\in\Nats \st
		\sum_{i = 1}^{r-1} g_i \le \ell <
	\sum_{i=1}^r g_i\bigr\},
	\]
	so that $\{\Gamma_r\}_{r\ge 1}$ is a partition of $\Nats$.
	For every $\ell\in \Gamma_r$,
	let $\alpha_\ell \defas \frac{1}{g_r2^r}$. 
	Observe that  for each $r\ge1$, we have
	$|\Gamma_r| = g_r$, and so
			$\sum_{\ell\in\Nats}\alpha_{\ell} = 1$.

As a consequence, there is a Borel map $\pi\colon [0,1]\to \Nats$ such that 
	for all $a\in\Rado_k$, we have
	\[
		\lambda(\pi^{-1}(a)) = \frac{\alpha_\ell}{|A_\ell|}
		\]
	where $\ell$ is such that $a\in A_\ell$.
	In other words, vertices in $\Rado_k$ of the same generation $A_\ell$
	are blown up to sets of the same positive measure, and the entire generation $A_\ell$ is blown up to a set of measure $\alpha_\ell$.
We may therefore apply Lemma~\ref{Bound on conditional entropy}.
Hence for any $\rho\colon[n] \to \Nats$, we have 
\[
	\RandomEntropyFinBig{
		G(n, W)
		\givenBig \bigwedge_{j \in [n]} \pi(\zeta_{\{j\}}) \in \Alice_{\rho(j)}} \geq {{|\rho([n])|}\choose {k}}
\]
a.s.

By Lemma~\ref{cond-reduces-entropy}, we have
\begin{align*}
\label{Random-free has arbitrary growth rates: Equation 1}
	\Entropy{G(\Nats, W)}(n) \ &\geq \  \Exp\Bigl(\RandomEntropyFinBig{G(n, W) \givenBig \bigvee_{\rho\colon[n] \to \Nats}
	\Bigl(\bigwedge_{j \in [n]} \pi(\zeta_{\{j\}}) \in \Alice_{\rho(j)}\Bigr) 
	\Bigr)
	}
\\	
	&\geq\sum_{\rho\colon[n] \to \Nats} \Pr
	\Bigl(\bigwedge_{j \in [n]} \pi(\zeta_{\{j\}}) \in \Alice_{\rho(j)}\Bigr) 
	\cdot 
	{{|\rho([n])|}\choose {k}}\\
\ &\geq \ 
	\Pr\Bigl(\bigl|Z|
	\geq n\cdot 2^{-q-2}\Bigr) \cdot 
{{\lceil n\cdot 2^{-q-2} \rceil}\choose {k}},
\tag{$\ddag$}
\end{align*}
where we define the random set $Z \defas \bigcup_{j\in [n]}\{\ell \in \Nats \st \pi(\zeta_{\{j\}}) \in \Alice_\ell \}$. 

Now define the random quantity 
	$X \defas |\{Z \cap \Gamma_q ]\}|$, and note that we always have $X \leq |Z|$. 
Because $\<\zeta_{\{j \}}\>_{j \in [n]}$ is an i.i.d.\ uniform sequence, we have
\begin{align*}
	\Exp(X)  \ &=  \sum_{\ell \in \Gamma_q} \Pr\Bigl(\bigvee_{j \in [n]} \pi(\zeta_{\{j \}}) \in \Alice_\ell\Bigr) \\
\ &=  \sum_{\ell \in  \Gamma_q} (1 - (1-\alpha_\ell)^n)\\
	\ &= g_q \cdot \left(1 - \left(1-\frac{1}{g_q 2^q}\right)^n\right).
\end{align*}

Observe that $(1- x)^n \ \leq\  1 - nx + n^2x^2 \ \leq\  1-\frac{nx}{2}$
holds for all
$x \in [0, \frac{1}{2n}]$.
Since $n \le g_q$ and $q\ge 1$, we have
$\frac{1}{g_q 2^q} \in [0, \frac{1}{2n}]$, and
so 
\[
1 - \left(1-\frac{1}{g_q 2^q}\right)^n \ \geq\  \frac{n}{g_q 2^{q+1}}.
\]
Putting these together, we get 
\[
\Exp(X) \ \geq\  n\cdot 2^{-q-1}.
\]
By Chebyshev's inequality, we have 
\[
	\Pr\left[|X - \Exp(X)|  \geq \frac{\Exp(X)}{2}\right]  \le
\frac{4 \Var(X)}{(\Exp(X))^2}.
\]
Hence,
\begin{align*}
1 - \frac{4 \Var(X)}{(\Exp(X))^2}
	\ &\le\ 
	 \Pr\left[|X - \Exp(X)| < \frac{\Exp(X)}{2}\right] \\
\ &\le\  \Pr\left[X > \frac{\Exp(X)}{2} \right]\\
\ &\le\  \Pr(X > n\cdot 2^{-q-2}) \\
\ &\le\  \Pr(|Z| > n\cdot 2^{-q-2}).
\end{align*}

For distinct $\ell, \ell' \in \Gamma_q$,
the events $\ell \in Z$ and $\ell' \in Z$ 
have negative correlation, 
which implies that
$\Var(X) \leq \Exp(X)$.
Hence
\[
\Pr(|Z| \ > \  n\cdot 2^{-q-2}) \ \geq\  1 - \frac{4}{\Exp(X)} \ \geq\  1 - \frac{4}{n\cdot 2^{-q-1}} \ \geq\  \frac{1}{2}.
\]

Finally, substituting in 
\eqref{Random-free has arbitrary growth rates: Equation 1} and recalling
that, by \eqref{ellstar} for $r = q$, we have $n \cdot 2^{-q-2} > k$ and $\gamma(n) \le 2^{-qk-3k-1} k^{-k}$, 
we obtain
\begin{align*}
	\Entropy{G(\Nats, W)}(n) \ &\geq\  \frac{1}{2} \cdot 
{{\lceil n\cdot 2^{-q-2} \rceil}\choose {k}} \\
	\ &\geq\  \frac{1}{2} \cdot \frac{(n\cdot 2^{-q-2} - k)^k}{k^k} \\
	\ &\geq\  \frac{1}{2} \cdot (n\cdot 2^{-q-3})^k k^{-k} \\
	\ &=\ n^k \cdot 2^{-qk-3k-1} k^{-k} \\
	\ &\geq\  n^k \cdot \gamma(n),
\end{align*}
as desired.
\end{proof}

\section{Non-redundant invariant measures}
\label{nonredundant-sec}
In this section we consider entropy functions of invariant measures for countable languages
that may be of unbounded arity. 
If an invariant measure fails to be non-redundant, then its entropy function may take the value $\infty$ even when
there are only finitely many relation symbols of each arity in the language.
Hence we restrict to the case of 
non-redundant invariant measures 
for relational languages, and
provide an upper bound on the entropy function in terms of the number of relation symbols 
of each arity.

\emph{For the rest of this section, let $\Lang$ be a countable relational language (possibly infinite).}

We show, in Proposition~\ref{nonred-bound}, that
the entropy function of a non-redundant invariant measure for $\Lang$
is dominated by that of a particular random $\Lang$-structure that generalizes the \ER\ random graph having edge probability $\frac12$.
We also calculate, in Lemma~\ref{nonred-binomial},
the entropy function of such a maximal entropy structure
explicitly in terms of the number of relation symbols of each arity in $\Lang$.
In the case where $\Lang$ has finitely many relation symbols of each arity, this provides a more precise 
version of Lemma~\ref{non-redundant-makes-it-real-valued}, which states
that such an entropy function takes values in $\Reals$.
Moreover, this calculation shows that there are $\Reals$-valued entropy functions that grow arbitrarily fast, in contrast to the situation for finite languages, where the growth is at most polynomial.

We now define the 
\emph{uniform non-redundant measure} for $\Lang$, which is the distribution of a random structure obtained by independently flipping a fair coin to decide every relation on a tuple of distinct elements, and setting all relations on tuples with repeated elements to false. 

For a relation symbol $R\in\Lang$, write $\arity(R)$ to denote its arity.

\begin{definition}
Given a set $X$, 
define 
\[\NonRepeating_{\Lang, X} \defas \{ \<R, \xx\> \st
R \in \Lang,\  \xx\in X\text{ has distinct entries, and}\ |\xx| = \arity(R)
\}.
\]
Let $\{\xi_{R(\xx)}\st \<R, \xx\> \in \NonRepeating_{\Lang, \Nats}
\}$ be a collection of i.i.d.\ uniform $\{\top, \bot\}$-valued random variables and let $\Xi$ be the $\Str_\Lang(\Nats)$-valued random variable given by
\[
\Xi \models R(\xx)\qquad \text{if and only if} 
\qquad
	\<R, \xx\> \in \NonRepeating_{\Lang, \Nats} \quad \text{and} \quad
	\xi_{R(\xx)} = \top
\]
for every $R\in\Lang$ and $\arity(R)$-tuple $\xx$ of elements from $\Nats$.
Define  $\UniformMeasureLang$ to be the distribution of $\Xi$, and call it
the \defn{uniform non-redundant measure} for $\Lang$. 
\end{definition}

It is easy to see that $\UniformMeasureLang$ is both invariant and non-redundant. 
Note that for $n\in\Nats$, the measure $(\UniformMeasureLang)_n$ is the uniform distribution on non-redundant structures in $\Str_\Lang(n)$.

	Let $\ArityLang\colon \Nats\to\Nats\cup \{\infty\}$ be the function sending each $n\in\Nats$
to the number of relation symbols in $\Lang$ having arity $n$.
In the following lemma we calculate the entropy function of $\UniformMeasureLang$ in terms of the function $\ArityLang$.

\begin{lemma}
	\label{nonred-binomial}
For any $n \in \Nats$, we have
\[
\Entropy{\UniformMeasureLang}(n) \ = \ \sum_{r \leq n} {n \choose r} \cdot r! \cdot \ArityLang(r).
\]
\end{lemma}
\begin{proof}
The invariant measure $\UniformMeasureLang$ is non-redundant,  and so for each $n\in\Nats$, the invariant measure
	$(\UniformMeasureLang)_n$ is concentrated on elements of $\Str_\Lang(n)$ in which no relations of arity greater than $n$ hold.
	Hence  $(\UniformMeasureLang)_n$ is determined by the set 
	$\{\xi_{R(\xx)}\st \<R, \xx\> \in \NonRepeating_{\Lang, [n]}\}$ of 
	random variables.
Because these random variables 
	are independent, we have 
\begin{align*}
\Entropy{\UniformMeasureLang}(n) \ &= 
\sum_{(R, \xx) \in \NonRepeating_{\Lang, [n]}}
\EntropyFin{\xi_{R(\xx)}}
\\
&= \ \bigl|\NonRepeating_{\Lang, [n]}\bigr| \\
&= \ \sum_{r \leq n} {n \choose r}
\cdot r! \cdot \ArityLang(r),
\end{align*}
as 
	the number of $r$-tuples from $[n]$ consisting of distinct elements 
	is ${n \choose r} \cdot r!$, 
	and there are $\ArityLang(r)$-many relation symbols of arity $r$ in $L$.
\end{proof}

Observe that when $L$ has only finitely many relation symbols of each arity, 
Lemma~\ref{nonred-binomial} shows that
$\Entropy{\UniformMeasureLang}$ is $\Reals$-valued.
Hence, by varying the choice of such an $L$ (and hence the function $\ArityLang$), 
we can obtain
$\Reals$-valued entropy functions
that grow arbitrarily fast.

We now show that $\UniformMeasureLang$ has the fastest growing entropy function among non-redundant invariant measures on $\Str_\Lang(\Nats)$.

\begin{proposition}
	\label{nonred-bound}
Let $\nu$ be a non-redundant invariant measure on $\Str_\Lang(\Nats)$. Then 
\[
\Entropy{\nu}(n) \ \leq \  \Entropy{\UniformMeasureLang}(n)
\]
for all $n \in \Nats$.
\end{proposition}
\begin{proof}
	Suppose $\cN$
	is a random $\Str_\Lang(\Nats)$-structure with distribution $\nu$. For $R \in \Lang$ and $\xx \in \Nats$ with $|\xx| = \arity(R)$, 
	recall that the random instantiation $R^\cN$ satisfies
\[
{R^\cN(\xx)} = \top
\qquad
\text{if and only if}
\qquad
\cN\models R(\xx).
\]

By Lemma~\ref{joint-entropy-less-than-sum-entropy} we then have, for $n \in \Nats$,
\[
\Entropy{\nu}(n) \ = \   \Entropy{\cN}(n) \ \leq  
	\sum_{\substack{R \in \Lang, \ \xx \in [n], \text{ and }\\|\xx| = \arity(R)}} \EntropyFinbig{{R^\cN(\xx)}}.
\]
If $\xx$ has duplicate entries then we know that $\EntropyFinbig{R^\cN(\xx)} = 0$ as $\cN$ is non-redundant a.s. Further, if $\xx$ has no duplicate entries then $\EntropyFinbig{{R^\cN(\xx)}} \leq \EntropyFin{\xi_{R(\xx)}}$ as the distribution of $\xi_{R(\xx)}$ is uniform on $\{\top, \bot\}$, and this is the distribution with maximal entropy on $\{\top, \bot\}$.

Therefore, we have 
	\begin{align*}
	\sum_{\substack{R \in \Lang, \ \xx \in [n], \text{ and }\\|\xx| = \arity(R)}} \EntropyFinbig{{R^\cN(\xx)}}
		\ &\leq \sum_{(R, \xx) \in \NonRepeating_{\Lang, [n]}} \hspace*{-10pt} \EntropyFin{\xi_{R(\xx)}}\\
		\ &=\   \Entropy{\Xi}(n) \\
		\ &=\   \Entropy{\UniformMeasureLang}(n),
	\end{align*}
	and so $\Entropy{\nu}(n) \ \le \ \Entropy{\UniformMeasureLang}(n),$
as desired.
\end{proof}

Putting together the previous lemma and proposition we immediately obtain the following bound.
\begin{corollary}
	\label{strongLemma}
Let $\nu$ be a non-redundant invariant measure on $\Str_\Lang(\Nats)$.
	Then 
	\[
		\Entropy{\nu}(n) \leq
	\sum_{r \leq n} {n \choose r} \cdot r! \cdot \ArityLang(r)\]
for all $n \in \Nats$.
	In particular, when $L$ is finite with maximum arity $k$, we have $\Entropy{\nu}(n) = O(n^k)$.
\end{corollary}

Recall from Lemma~\ref{non-redundant-makes-it-real-valued} that when 
$L$ has only finitely many relation symbols of each arity, any non-redundant invariant measure on $\Str_\Lang(\Nats)$
has an $\Reals$-valued entropy function. Corollary~\ref{strongLemma} improves this by providing an explicit upper bound.

In fact, we do not need the invariant measure to be non-redundant, nor the language to be relational, to obtain the final line of Corollary~\ref{strongLemma};
by Proposition~\ref{nonredundant-entropy-function}
we see that the polynomial bound $O(n^k)$ holds for the entropy function of an arbitrary invariant measure for a finite language (possibly with constant or function symbols) of maximum arity $k$, thereby strengthening Lemma~\ref{realvalued-lemma}.

An interesting question for future work is to characterize precisely those functions from $\Nats$ to $\Reals \cup \{\infty\}$ that can be the entropy function of an invariant measure. 
By Proposition~\ref{nonredundant-entropy-function}, it suffices to consider entropy functions of non-redundant invariant measures for relational languages.

\section*{Acknowledgements}
The authors thank Jan Reimann and Daniel M.\ Roy for helpful conversations, and the anonymous referee for their comments.



\begin{thebibliography}{AFKwP17}

\bibitem[Ack15]{AckermanAutM}
N.~Ackerman, \emph{Representations of {A}ut({M})-invariant measures: Part 1},
  arXiv e-print 1509.06170v1 (2015).

\bibitem[AFKrP17]{2017arXiv171009336A}
N.~{Ackerman}, C.~{Freer}, A.~{{Kr}uckman}, and R.~{Patel}, \emph{{Properly
  ergodic structures}}, ArXiv e-print 1710.09336 (2017).

\bibitem[AFKwP17]{MR3564374}
N.~Ackerman, C.~Freer, A.~{{{Kw}}iatkowska}, and R.~Patel, \emph{A
  classification of orbits admitting a unique invariant measure}, Ann. Pure
  Appl. Logic \textbf{168} (2017), no.~1, 19--36.

\bibitem[AFNP16]{MR3425980}
N.~Ackerman, C.~Freer, J.~Ne\v{s}et\v{r}il, and R.~Patel, \emph{Invariant
  measures via inverse limits of finite structures}, European J. Combin.
  \textbf{52} (2016), 248--289.

\bibitem[AFP16]{MR3515800}
N.~Ackerman, C.~Freer, and R.~Patel, \emph{Invariant measures concentrated on
  countable structures}, Forum Math. Sigma \textbf{4} (2016), e17, 59pp.

\bibitem[AFP17]{AFPcompleteclassification}
\bysame, \emph{Countable infinitary theories
  admitting an invariant measure}, arXiv e-print 1710.06128v1 (2017).

\bibitem[Ald81]{MR637937}
D.~J. Aldous, \emph{Representations for partially exchangeable arrays of random
  variables}, J. Multivariate Anal. \textbf{11} (1981), no.~4, 581--598.

\bibitem[Ald85]{MR883646}
\bysame, \emph{Exchangeability and related topics}, \'Ecole d'\'et\'e de
  probabilit\'es de {S}aint-{F}lour, {XIII} (1983), Lecture Notes in Math.,
  vol. 1117, Springer, Berlin, 1985, pp.~1--198.

\bibitem[Aus08]{MR2426176}
T.~Austin, \emph{On exchangeable random variables and the statistics of large
  graphs and hypergraphs}, Probab. Surv. \textbf{5} (2008), 80--145.

\bibitem[AZ86]{MR831437}
G.~Ahlbrandt and M.~Ziegler, \emph{Quasi-finitely axiomatizable totally
  categorical theories}, Ann. Pure Appl. Logic \textbf{30} (1986), no.~1,
  63--82.

\bibitem[CD13]{MR3127871}
S.~Chatterjee and P.~Diaconis, \emph{Estimating and understanding exponential
  random graph models}, Ann. Statist. \textbf{41} (2013), no.~5, 2428--2461.

\bibitem[CT06]{MR2239987}
T.~M. Cover and J.~A. Thomas, \emph{Elements of information theory}, 2nd ed.,
  Wiley, Hoboken, NJ, 2006.

\bibitem[CV11]{MR2825532}
S.~Chatterjee and S.~R.~S. Varadhan, \emph{The large deviation principle for
  the {E}rd{\H{o}}s-{R}\'enyi random graph}, European J. Combin. \textbf{32}
  (2011), no.~7, 1000--1017.

\bibitem[DF81]{MR640207}
P.~Diaconis and D.~Freedman, \emph{On the statistics of vision: the {J}ulesz
  conjecture}, J. Math. Psych. \textbf{24} (1981), no.~2, 112--138.

\bibitem[DJ08]{MR2463439}
P.~Diaconis and S.~Janson, \emph{Graph limits and exchangeable random graphs},
  Rend. Mat. Appl. (7) \textbf{28} (2008), no.~1, 33--61.

\bibitem[HJS18]{MR3742179}
H.~Hatami, S.~Janson, and B.~Szegedy, \emph{Graph properties, graph limits, and
  entropy}, J. Graph Theory \textbf{87} (2018), no.~2, 208--229.

\bibitem[HN13]{MR3073488}
H.~Hatami and S.~Norine, \emph{The entropy of random-free graphons and
  properties}, Combin. Probab. Comput. \textbf{22} (2013), no.~4, 517--526.

\bibitem[Jan13]{MR3043217}
S.~Janson, \emph{Graphons, cut norm and distance, couplings and
  rearrangements}, New York Journal of Mathematics (NYJM) Monographs, vol.~4,
  State University of New York, University at Albany, 2013.

\bibitem[Kal99]{MR1702867}
O.~Kallenberg, \emph{Multivariate sampling and the estimation problem for
  exchangeable arrays}, J. Theoret. Probab. \textbf{12} (1999), no.~3,
  859--883.

\bibitem[Kal02]{MR1876169}
\bysame, \emph{Foundations of modern probability}, 2nd ed., Probability and its
  Applications, Springer, New York, 2002.

\bibitem[Kal05]{MR2161313}
\bysame, \emph{Probabilistic symmetries and invariance principles}, Probability
  and its Applications, Springer, New York, 2005.

\bibitem[Kec95]{MR1321597}
A.~S. Kechris, \emph{Classical descriptive set theory}, Graduate Texts in
  Mathematics, vol. 156, Springer-Verlag, New York, 1995.

\bibitem[Lov12]{MR3012035}
L.~Lov\'asz, \emph{Large networks and graph limits}, American Mathematical
  Society Colloquium Publications, vol.~60, American Mathematical Society,
  Providence, RI, 2012.

\bibitem[LS10]{MR2815610}
L.~Lov\'asz and B.~Szegedy, \emph{Regularity partitions and the topology of
  graphons}, An irregular mind, Bolyai Soc. Math. Stud., vol.~21, J\'anos
  Bolyai Math. Soc., Budapest, 2010, pp.~415--446.

\bibitem[PV10]{MR2724668}
F.~Petrov and A.~Vershik, \emph{Uncountable graphs and invariant measures on
  the set of universal countable graphs}, Random Structures \& Algorithms
  \textbf{37} (2010), no.~3, 389--406.

\bibitem[RS13]{MR3083277}
C.~Radin and L.~Sadun, \emph{Phase transitions in a complex network}, J. Phys.
  A \textbf{46} (2013), no.~30, 305002.

\end{thebibliography}

\providecommand{\bysame}{\leavevmode\hbox to3em{\hrulefill}\thinspace}
\providecommand{\MR}{\relax\ifhmode\unskip\space\fi MR }
\providecommand{\MRhref}[2]{%
  \href{http://www.ams.org/mathscinet-getitem?mr=#1}{#2}
}
\providecommand{\href}[2]{#2}

\end{document}